\documentclass[11pt]{amsart}
\usepackage{amsmath}
\theoremstyle{plain}
\newtheorem{theorem}{Theorem}[section]
\newtheorem{lemma}[theorem]{Lemma}

\theoremstyle{definition}
\newtheorem{remark}[theorem]{Remark}

\numberwithin{equation}{section}

\usepackage{tikz}

\def\be{\begin{equation}}
\def\ee{\end{equation}}

\begin{document}

\title[Liouville's Equation in Singular Domains]
{Boundary Behaviors for Liouville's Equation\\ in Planar Singular Domains}
\author{Qing Han}
\address{Department of Mathematics\\
University of Notre Dame\\
Notre Dame, IN 46556, USA} \email{qhan@nd.edu}
\author{Weiming Shen}
\address{Beijing International Center for Mathematical Research\\
Peking University\\
Beijing, 100871, China}  \email{wmshen@pku.edu.cn}

\begin{abstract}
We study asymptotic behaviors near the boundary of complete metrics of constant curvature
in planar singular domains and establish an optimal estimate of these metrics by the
corresponding metrics in tangent cones near isolated singular points on boundary.
The conformal structure plays an essential role.
\end{abstract}

\thanks{The first author acknowledges the support of NSF
Grant DMS-1404596. The second author acknowledges the support of NSFC
Grant 11571019.}
\maketitle

\section{Introduction}\label{sec-Intro}

Assume $\Omega\subset \mathbb{R}^{2}$ is a domain. We consider the following problem:
\begin{align}
\label{eq-MainEq} \Delta{u}& =e^{ 2u } \quad\text{in }\Omega, \\
\label{eq-MainBoundary}u&=\infty\quad\text{on }\partial \Omega.
\end{align}
The equation \eqref{eq-MainEq} is known as Liouville's equation.
For a large class of domains $\Omega$, \eqref{eq-MainEq} and \eqref{eq-MainBoundary} admit a
solution $u\in C^\infty(\Omega)$.
Geometrically,  $e^{ 2u }(dx_1\otimes dx_1+dx_2\otimes dx_2)$ is
a complete metric with constant Gauss curvature $-1$ on $\Omega$.
Our main concern in this paper is the
asymptotic behavior of solutions $u$ near isolated {\it singular} points on boundary.

The higher dimensional counterpart
is given by, for $\Omega\subset\mathbb R^n$, $n\ge 3$,
\begin{align}
\label{eq-MainEq-HigherDim} \Delta{u} = \frac14n(n-2)u^{\frac{n+2}{n-2}}\quad\text{in }\Omega.
\end{align}
Geometrically, $u^{\frac{4}{n-2}}\sum_{i=1}^{n}dx_i\otimes dx_i$ is
a complete metric with the constant scalar curvature $-n(n-1)$ on $\Omega$.
More generally, we can study, for a function $f$,
\begin{align}
\label{eq-MainEq-HigherDim-general} \Delta{u} = f(u)\quad\text{in }\Omega.
\end{align}

The study of these problems has a rich history.
Bieberbach \cite{Bieberbach1916} studied the problem \eqref{eq-MainEq} and \eqref{eq-MainBoundary} and
Loewner and Nirenberg \cite{Loewner&Nirenberg1974} studied the problem
\eqref{eq-MainEq-HigherDim} and \eqref{eq-MainBoundary}. They proved that
there exists a solution in every bounded domain satisfying the inner and outer sphere condition.
Lazer and McKenna \cite{LazerMcKenna1993} proved the existence of solutions of
\eqref{eq-MainEq} and \eqref{eq-MainBoundary} in domains satisfying the outer sphere condition.
If $f$ is monotone, Keller \cite{Keller1957} established the existence for
\eqref{eq-MainEq-HigherDim-general} and \eqref{eq-MainBoundary}. We refer to the survey paper
\cite{Bandle&Flucher1996}
for more information on the equations  \eqref{eq-MainEq} and \eqref{eq-MainEq-HigherDim}.

In a pioneering work,
Loewner and Nirenberg \cite{Loewner&Nirenberg1974} studied asymptotic behaviors
of solutions of \eqref{eq-MainEq-HigherDim} and \eqref{eq-MainBoundary} and proved
an estimate involving leading terms.
A similar estimate can be established for solutions of \eqref{eq-MainEq} and \eqref{eq-MainBoundary}.
Kichenassamy
\cite{Kichenassamy2004JFA}, \cite{Kichenassamy2005JFA} expanded further if $\Omega$ has a $C^{2,\alpha}$-boundary.
See also Bandle and Marcus \cite{Bandle&Marcus1995} and
Diaz and Letelier \cite{DiazLetelier1992}, for example, for more general $f$.
Moreover, if $\Omega$ has a smooth boundary, an
estimate up to an arbitrarily finite order was established by
Andersson, Chru\'sciel and Friedrich \cite{ACF1982CMP} and Mazzeo \cite{Mazzeo1991}.
In fact, they proved that solutions of \eqref{eq-MainEq-HigherDim} and \eqref{eq-MainBoundary}
are polyhomogeneous.
All these results require $\partial\Omega$ to have
some degree of regularity. The case where $\partial\Omega$ is singular was
studied by del Pino and Letelier \cite{delPino2002}
and Marcus and Veron \cite{Marcus&Veron1997}.
However, no explicit estimates are known in neighborhoods of singular boundary points.

Other problems with a similar feature include complete K\"ahler-Einstein metrics discussed by
Cheng and Yau \cite{ChengYau1980CPAM},
Fefferman \cite{Fefferman1976}, and Lee and Melrose \cite{LeeMelrose1982},
the complete minimal graphs in the hyperbolic space by Han and Jiang \cite{HanJiang},
Lin \cite{Lin1989Invent} and Tonegawa \cite{Tonegawa1996MathZ} and
a class of Monge-Amp\`{e}re equations by Jian and Wang \cite{JianWang2013JDG}.

Now we return to \eqref{eq-MainEq}-\eqref{eq-MainBoundary} and study behaviors of solutions near boundary.
For bounded domains $\Omega\subset\mathbb R^2$, let $d$ be the distance function to $\partial\Omega$.
Then, $d$ near $\partial\Omega$ has the same regularity as $\partial\Omega$.
We first employ a blowup process
to make a reasonable guess of the form of leading terms.
When the domain is at least $C^1$,
blowing up at a boundary point yields a half space whose boundary is the tangent line of the original domain. The solution
over the half-space, say $\{(x_1,x_2): x_2>0\}$, is given by the one-variable function $-\log x_2$.
However, the half-space  may not match
the original domain locally at the blowup point. Therefore, we need to modify this one-variable function.
Due to the regularity of the boundary, the function $-\log d$, defined in the original domain,
is a good replacement of $-\log x_2$.
This informal discussion provides a good reasoning that
$-\log d$ should be the leading term when the domain is more than $C^1$.
Although $u$ is smooth inside the domain due to the regularity
of solutions to elliptic equations,
the leading term $-\log d$
has the same regularity as the boundary and is
not always smooth near the boundary.

In fact, under the condition that $\partial\Omega$ is $C^2$,
the solution $u$ of
\eqref{eq-MainEq}-\eqref{eq-MainBoundary} satisfies
\begin{equation}\label{eq-EstimateDegree1}|u+\log d|\le Cd,\end{equation}
where $C$ is a positive constant depending only on the geometry of $\partial\Omega$.
(See Theorem \ref{thrm-C-1,alpha-expansion}.)
The  proof of \eqref{eq-EstimateDegree1} is by the maximum principle,
specifically, by a comparison of $u$ and the corresponding solutions
in the interior tangent balls and outside the exterior tangent balls, respectively.
We also point out that $Cd$ in the right-hand side is optimal
under the assumption that the boundary is $C^2$. Refer to
\cite{Loewner&Nirenberg1974} for a similar estimate for solutions of
\eqref{eq-MainEq-HigherDim} and \eqref{eq-MainBoundary}.

In this paper, we study
asymptotic behaviors of $u$ near isolated singular points on $\partial\Omega$.
We first describe our setting and employ the blowup process as above to make a reasonable guess
of the form of the leading terms.

Taking a boundary point, say the origin, we assume $\partial\Omega$
has a conic singularity at the origin in the following sense: $\partial\Omega$ in a neighborhood
of the origin consists of two $C^2$-curves $\sigma_1$ and $\sigma_2$,
intersecting at the origin with an angle $\mu\pi$ for some constant $\mu\in (0,2)$. Here, the origin is
an end point of the both curves $\sigma_1$ and $\sigma_2$. Let $l_1$ and $l_2$ be two rays
starting from the origin and tangent to $\sigma_1$ and $\sigma_2$ there, respectively.
Then, an infinite cone $V_{\mu}$ formed by $l_1$ and $l_2$ is considered as a tangent cone
of $\Omega$ at the origin, with an opening angle $\mu\pi$. Solutions of
\eqref{eq-MainEq}-\eqref{eq-MainBoundary} in $V_{\mu}$ can be written explicitly.
In fact, using polar coordinates, we write
$$V_{\mu}=\{(r,\theta):\, r\in (0,\infty),\, \theta\in (0,\mu\pi)\}.$$
Here, $l_1$ corresponds to $\theta=0$ and $l_2$ to $\theta=\mu\pi$.
Then, the solution $v_{\mu}$ of \eqref{eq-MainEq}-\eqref{eq-MainBoundary} in $V_\mu$
is given by
\begin{equation}\label{eq-definition_v_main}v_{\mu}=-\log\left(\mu r\sin\frac\theta\mu\right).\end{equation}
The blowup process suggests that $v_{ \mu}$ should provide a good approximation of $u$ near the origin.
However, we encounter the same problem that
the tangent cone may not match
the original domain locally at the blowup point.
For a remedy, we need to modify $v_\mu$ to get a function defined in
$\Omega$ near the origin.

Let  $d, d_1$ and $ d_2$ be the distances to $\partial\Omega, \sigma_1$ and $ \sigma_2$, respectively. Then,
$d=\min\{d_1, d_2\}$ near the origin.
For $\mu\in (0,1]$, we define, for any $x\in \Omega$,
\begin{equation}\label{eq-definition_f_main1}f_{\mu}(x)=
-\log \left(\mu |x|  \sin\frac{\arcsin\frac{d(x)}{|x|}}{\mu}\right).\end{equation}
We note that $f_\mu$ in \eqref{eq-definition_f_main1} is well-defined for $x$ sufficiently small and that
$\{x\in\Omega:\, d_1(x)=d_2(x)\}$ is a curve from the origin for $\mu\in (0,1]$ near the origin.
The case $\mu\in (1,2)$ is slightly more complicated since $\{x\in\Omega:\, d_1(x)=d_2(x)\}$ has a nonempty interior
and $d(x)=|x|$ there.
We can still use \eqref{eq-definition_f_main1} to define $f_\mu(x)$ for $x\in\Omega$ with $d_1(x)\neq d_2(x)$
and we need to modify for $x\in\Omega$ with $d_1(x)=d_2(x)$. We will provide such a modification in Section \ref{sec-IsolatedSingular},
specifically by \eqref{eq-definition_f}.

We now state our main result in this paper.

\begin{theorem}\label{thrm-Main}
Let $\Omega$ be a bounded domain in $\mathbb R^2$
and $\partial\Omega\cap B_{r_0}$ consist of two $C^{2}$-curves
$\sigma_1$ and $\sigma_2$ intersecting at the origin at an angle $\mu\pi$, for some
constant $\mu\in (0,2)$ and some $r_0>0$.
Suppose $ u \in C^{2}(\Omega)$ is  a solution of
\eqref{eq-MainEq}-\eqref{eq-MainBoundary}.
Then, for any $x\in\Omega\cap B_\delta$,
\begin{equation}\label{eq-MainEstimate}\left|u(x)
-f_\mu(x)\right|\le Cd(x),\end{equation}
where $f_\mu$ is the function defined in \eqref{eq-definition_f_main1} for
$\mu\in (0,1]$ and in \eqref{eq-definition_f} for $\mu\in (1,2)$, $d$ is the distance to $\partial\Omega$,
and $\delta$ and $C$ are positive constants depending only on $\mu$, $r_0$ and the
$C^2$-norms of $\sigma_1$ and $\sigma_2$.
\end{theorem}

The estimate \eqref{eq-MainEstimate} generalizes \eqref{eq-EstimateDegree1}
to singular domains and is optimal. The power one of the
distance function in the right-hand side cannot be improved without better regularity
assumptions of the boundary. The proof of Theorem \ref{thrm-Main}
is based on a combination of conformal transforms and the maximum principle. An appropriate
conformal transform changes the tangent cone at the origin to the upper half plane.
The new boundary has a better regularity at the origin for $\mu\in (1,2)$ and becomes
worse for $\mu\in (0,1)$. Such a change in the regularity of the boundary requires us to
discuss asymptotic behaviors of solutions near $C^{1,\alpha}$-boundary and near
$C^{2,\alpha}$-boundary.

The paper is organized as follows. In Section \ref{sec-Existence},
we provide some preliminaries for
solutions of \eqref{eq-MainEq}-\eqref{eq-MainBoundary}.
In Section \ref{sec-C-1,alpha-boundary}, we study the asymptotic
expansions near $C^{1,\alpha}$-boundary and derive an optimal estimate.
In Section \ref{sec-C-2,alpha-boundary}, we study the asymptotic
expansions near $C^{2,\alpha}$-boundary and derive the corresponding optimal estimate.
In Section \ref{sec-IsolatedSingular}, we study asymptotic behaviors near isolated
singular points and prove Theorem \ref{thrm-Main}.

We would like to thank Matthew Gursky for suggesting the problem
to investigate asymptotic behaviors of solutions of
\eqref{eq-MainEq-HigherDim} and \eqref{eq-MainBoundary} near singular boundary points,
a project we will pursue elsewhere.

\section{Preliminaries}\label{sec-Existence}

In this section,
we collect some well-known results concerning
solutions of \eqref{eq-MainEq}-\eqref{eq-MainBoundary}.


Let $x_{0}\in \mathbb R^2$ be a point and $r>0$ be a constant.
For $\Omega=B_r(x_0)$, denote by $u_{r,x_0}$ the corresponding solution of
\eqref{eq-MainEq}-\eqref{eq-MainBoundary}. Then,
\begin{equation}\label{eq-solution-inside}u_{r,x_{0}}(x)=\log\frac{2r}{r^{2}-|x-x_{0}|^{2}}.\end{equation}
With $d(x)=r-|x-x_0|$, we have
$$u_{r,x_{0}}=-\log d-\log\left(1-\frac{d}{2r}\right).$$
For $\Omega=\mathbb R^2\setminus B_r(x_0)$, denote by $v_{r,x_0}$ the corresponding solution of
\eqref{eq-MainEq}-\eqref{eq-MainBoundary}. Then,
\begin{equation}\label{eq-solution-outside}v_{r,x_{0}}(x)=\log\frac{2r}{|x-x_{0}|^{2}-r^{2}}.\end{equation}
With $d(x)=|x-x_0|-r$, we have
$$v_{r,x_{0}}=-\log d-\log\left(1+\frac{d}{2r}\right).$$
These two solutions play an important role in this paper.

Now, we state the well-known existence and uniqueness of solutions
of \eqref{eq-MainEq}-\eqref{eq-MainBoundary}.
Refer to \cite{LazerMcKenna1993} for a proof.

\begin{theorem}\label{thrm-ExistenceUniqueness}
Let $ \Omega $ be a bounded domain in $\mathbb R^2$
satisfying a uniform exterior cone condition.
Then, there exists a unique solution $ u \in C^{\infty}(\Omega)$ of
\eqref{eq-MainEq}-\eqref{eq-MainBoundary}.
\end{theorem}

To end this section, we prove a preliminary result for domains with singularity.
We note that a finite cone
is determined by its vertex, its axis, its height and its opening angle.

\begin{lemma}\label{lemma-ExteriorCone}
Let $ \Omega $ be a bounded domain in $\mathbb R^2$
satisfying a uniform exterior cone condition.
Suppose $ u \in C^{2}(\Omega)$ is  a solution of
\eqref{eq-MainEq}-\eqref{eq-MainBoundary}.
Then, for any $x\in \Omega$ with $d(x)<\delta$,
$$|u(x)+\log d(x)|\le C,$$
where $\delta$ and $C$ are positive constants
depending only on the uniform exterior cone.
\end{lemma}

\begin{proof}
For any $ x \in \Omega$ with $d(x)=d$, we have $ B_{d}(x) \subset \Omega$.
We assume $d=|x-p|$ for some $p\in\partial\Omega $.
Let $u_{d,x}$ be the solution of \eqref{eq-MainEq}-\eqref{eq-MainBoundary}
in $B_d(x)$, given by \eqref{eq-solution-inside}.
By the maximum
principle, we have
$$u(x) \leq u_{d,x} (x)= -\log d-\log\left(1-\frac{d}{2d}\right)
=-\log d+\log 2.$$
Next, there exists a cone $V$, with vertex $p$,  axis $\overrightarrow{e_{p}}$, height $h$
and opening angle $2\theta$, such that $V\cap \Omega=\emptyset$. Here, we can assume $h$
and $\theta$
do not depend on the choice of $p \in \partial \Omega.$
Set $\widetilde{p}= p+\frac{1}{\sin\theta}d \overrightarrow{e_{p}}.$
It is straightforward to check $B_{d}(\widetilde{p})\subset V\subset\Omega^{C} $,
if $d<\frac{h}{1+\frac{1}{\sin\theta}}$, and
dist$(x,\partial B_d(\widetilde p))\le \frac{d}{\sin\theta}$.
Let $v_{d,\widetilde p}$ be the solution of \eqref{eq-MainEq}-\eqref{eq-MainBoundary}
in $\mathbb R^2\setminus B_d(\widetilde p)$, given by \eqref{eq-solution-outside}.
Then, by the maximum principle, we have
$$u(x) \geq v_{d,\widetilde{p}}(x)\geq
-\log\left(\frac{d}{\sin\theta}\right)
-\log\left(1+\frac{d}{2d\sin\theta}\right)
=-\log d-\log\left(\frac{1+2\sin\theta}{2\sin^2\theta}\right). $$
We have the desired result. \end{proof}

\section{Expansions near $C^{1,\alpha}$-boundary}\label{sec-C-1,alpha-boundary}

In this section, we study
asymptotic behaviors near $C^{1,\alpha}$-portions of  $\partial\Omega$.
A similar estimate for solutions of
\eqref{eq-MainEq-HigherDim} and \eqref{eq-MainBoundary} can be found in
\cite{Loewner&Nirenberg1974} under the $C^{1,1}$-assumption
of the boundary.

\begin{theorem}\label{thrm-C-1,alpha-expansion}
Let $\Omega$ be a bounded domain in $\mathbb R^2$ and $\partial\Omega\cap B_{r_0}(x_0)$ be
$C^{1,\alpha}$ for some $x_0\in\partial\Omega$, $r_0>0$ and $\alpha\in (0,1]$. Suppose
$u\in C^2(\Omega)$ is a solution of \eqref{eq-MainEq}-\eqref{eq-MainBoundary}. Then,
$$|u(x)+\log d(x)|\le Cd^\alpha(x)\quad\text{for any }x\in\Omega\cap B_r(x_0),$$
where $d(x)$ is the distance from $x$ to $\partial\Omega$, and $r$ and $C$
are positive constants depending only on $r_0$, $\alpha$ and the $C^{1,\alpha}$-norm of
$\partial\Omega$.
\end{theorem}

\begin{proof}
We take $R>0$ sufficiently small such that $\partial\Omega\cap B_{R}(x_0)$ is
$C^{1,\alpha}$.
We fix an $x\in\Omega\cap B_{R/4}(x_0)$ and take $p\in \partial\Omega$, also
near $x_0$, such that
$d(x)=|x-p|$.
Then,  $p\in  \partial\Omega\cap B_{R/2}(x_0)$.
By a translation and rotation, we assume $p=0$ and the $x_2$-axis is the interior normal to $\partial\Omega$
at 0. Then, $x$ is on the positive $x_2$-axis,
with $d=d(x)=|x|$, and the $x_1$-axis is the tangent line of $\partial\Omega$ at 0.
Moreover, a portion of $\partial \Omega$ near 0 can be expressed as a $C^{1,\alpha}$-function $\varphi$ of
$x_1\in (-s_0,s_0)$, with $\varphi(0)=0$, and
\begin{equation}\label{eq-boundaryC1alpha}
|\varphi(x_1)|\le M|x_1|^{1+\alpha}\quad\text{for any }x_1\in (-s_0,s_0).\end{equation}
Here, $s_0$ and $M$ are positive constants chosen to be uniform, independent of $x$.

We first consider the case $\alpha=1$. For any $r>0$,  the lower semi-circle of
$$x_1^2+(x_2-r)^2=r^2$$ satisfies
$x_2\ge x_1^2/(2r)$.
By fixing a constant $r$ sufficiently small,
\eqref{eq-boundaryC1alpha} implies
$$B_r(re_2)\subset\Omega\text{ and }B_r(-re_2)\cap \Omega=\emptyset.$$
Let $u_{r, re_2}$ and $v_{r, -re_2}$ be the solutions of
\eqref{eq-MainEq}-\eqref{eq-MainBoundary} in $B_r(re_2)$ and $\mathbb R^2\setminus B_r(-re_2)$,
given by \eqref{eq-solution-inside} and \eqref{eq-solution-outside}
respectively. Then, by the maximum principle, we have
$$v_{r,-re_2}\le u\le u_{r, re_2}\quad\text{in }B_r(re_2).$$
For the $x$ above in the positive $x_2$-axis with $|x|=d<r$, we obtain
$$-\log d-\log\left(1+\frac{d}{2r}\right)\le u\le -\log d-\log\left(1-\frac{d}{2r}\right).$$
This implies the desired result for $\alpha=1$.

Next, we consider $\alpha\in (0,1)$. Recall that $x$ is in the positive $x_2$-axis and $|x|=d$.
We first note
\begin{equation}\label{eq-boundaryC1alpha2}
|x_{1}|^{1+\alpha}\le d^{1+\alpha}+\frac{1}{d^{1-\alpha}}x_{1}^{2} \quad\text{for any }x_1\in\mathbb R.\end{equation}
This follows from the H\"older inequality, or more easily, by considering $|x_1|\le d$ and $|x_1|\ge d$
separately. Let $r=d^{1-\alpha}/(2M)$ and $q$ be the point on the positive $x_2$-axis such that
$|q|=Md^{1+\alpha}+r$. By taking $d$ sufficiently small, \eqref{eq-boundaryC1alpha}
and \eqref{eq-boundaryC1alpha2} imply
$$B_r(q)\subset\Omega\text{ and }B_r(-q)\cap \Omega=\emptyset.$$
Let $u_{r, q}$ and $v_{r, -q}$ be the solutions of
\eqref{eq-MainEq}-\eqref{eq-MainBoundary} in $B_r(q)$ and $\mathbb R^2\setminus B_r(-q)$,
given by \eqref{eq-solution-inside} and \eqref{eq-solution-outside}
respectively. Then, by the maximum principle, we have
$$v_{r,-q}\le u\le u_{r, q}\quad\text{in }B_r(q).$$
For the $x$ above,
dist$(x, \partial B_r(q))=d-Md^{1+\alpha}$ and
dist$(x, \partial B_r(-q))=d+Md^{1+\alpha}$. Evaluating at such an $x$, we obtain
\begin{align*}&-\log(d+Md^{1+\alpha})-\log\left(1+\frac{M}{d^{1-\alpha}}(d+Md^{1+\alpha})\right)\\
&\qquad
\le u\le -\log(d-Md^{1+\alpha})-\log\left(1-\frac{M}{d^{1-\alpha}}(d-Md^{1+\alpha})\right).\end{align*}
This implies the desired result for $\alpha\in (0,1)$. \end{proof}

\begin{remark} Domains are assumed to be bounded in Theorem \ref{thrm-C-1,alpha-expansion}
so that it is easier
to compare $u=\infty$ on $\partial\Omega$ with functions of finite values. Our main interest
is estimates near $x_0\in \partial\Omega$. A similar remark holds for some results
in the rest of the paper.
\end{remark}

\section{Expansions near $C^{2,\alpha}$-boundary}\label{sec-C-2,alpha-boundary}

In this section, we study
asymptotic behaviors near $C^{2,\alpha}$-portions of  $\partial\Omega$.
Under the condition of the $C^{2,\alpha}$-boundary, Kichenassamy  \cite{Kichenassamy2004JFA} proved that
$\exp(-u)$ is $C^{2,\alpha}$ up to the boundary by establishing Schauder
estimates for degenerate elliptic equations of Fuchsian type. Such a $C^{2,\alpha}$-regularity implies
an expansion up to order $1+\alpha$, which will be needed in this paper.
Since we only need this expansion, instead of the full $C^{2,\alpha}$-regularity in  \cite{Kichenassamy2004JFA}, we will
use the maximum principle to present a more direct proof, which is consistent with the proof of
Theorem \ref{thrm-C-1,alpha-expansion}.
It is straightforward to derive the upper bound and extra work is needed for lower bound.
We also note that the curvature of the boundary is only $C^\alpha$ in the present case
and hence cannot be differentiated.

\begin{theorem}\label{thrm-C-2,alpha-expansion}
Let $\Omega$ be a bounded domain in $\mathbb R^2$ and $\partial\Omega\cap B_{r_0}(x_0)$ be
$C^{2,\alpha}$ for some $x_0\in\partial\Omega$, $r_0>0$ and $\alpha\in (0,1)$. Suppose
$u\in C^2(\Omega)$ is a solution of \eqref{eq-MainEq}-\eqref{eq-MainBoundary}. Then,
$$\left|u(x)+\log d(x)-\frac{1}{2}\kappa(y) d(x)\right|\le Cd^{1+\alpha}(x)
\quad\text{for any }x\in\Omega\cap B_r(x_0),$$
where $d(x)$ is the distance from $x$ to $\partial\Omega$,
$\kappa(y)$ is the curvature of $\partial\Omega$ at $y\in\partial\Omega$ with
$|y-x|=d(x)$, and $r$ and $C$
are positive constants depending only on $r_0$, $\alpha$ and the $C^{2,\alpha}$-norm of
$\partial\Omega$.
\end{theorem}

\begin{proof}
We take $R>0$ sufficiently small such that $\partial\Omega\cap B_{2R}(x_0)$ is
$C^{2,\alpha}$ and that $d$ is $C^{2,\alpha}$ in $\Omega\cap B_{2R}(x_0)$.
The proof consists of several steps.

{\it Step 1.} Set
\begin{equation}
u=v-\log d.
\end{equation}
A straightforward calculation yields
\begin{equation}
\textit{S}(v)=0  \quad\text{in }\Omega,
\end{equation}
where
\begin{equation}\label{eq for v}
\textit{S}(v)=d\Delta v-\Delta d-\frac{1}{d}(e^{2v}-1).
\end{equation}
By Theorem \ref{thrm-C-1,alpha-expansion} for $\alpha=1$,
we have
$$|v|\leq C_{0}d\quad\text{in }\Omega\cap B_{R}(x_0),$$
for some constant $C_{0}$ depending only on the geometry of  $\Omega$.
In particular, $v=0$ on  $\partial\Omega\cap B_{R}(x_0)$.

To proceed, we denote by $(x', d)$ the principal coordinates in $\bar\Omega\cap B_R(x_0)$.
Then, 
$$\Delta v=\frac{\partial^{2}v}{\partial d^{2}} + G\frac{\partial^{2}v}{\partial x'^{2}}
+I_{x'} \frac{\partial v}{\partial x'}+I_d \frac{\partial v}{\partial d},$$
where $G$, $I_{x'}$ and $I_d$ are at least continuous functions in
$\bar\Omega\cap B_R(x_0)$. We note that $G$ has a positive lower bound
and $I_d$ has the form
\begin{equation}\label{laplace_d}
I_d = -\kappa +O(d^\alpha),
\end{equation}
where $\kappa$ is the curvature of $\partial\Omega$.
Set, for any constant $r>0$,
$$G_r =\{(x',d):|x'|\leq r, 0<d<r \}.$$

{\it Step 2.} We now construct supersolutions and prove an upper bound of $v$.
We set
\begin{equation}\label{eq-definition_w}w(x)=d(x'^{2}+d^2)^{\frac{\alpha}{2}},\end{equation}
and, for some positive constants $A$ and $B$ to be determined,
$$\overline{v}=\frac{1}{2}\kappa(0)d+ Aw+Bd^{1+\alpha}.$$
We write
$$\textit{S}(\overline{v})=d\Delta \overline{v}-\Delta d-\frac2d\overline{v}
-\frac1d(e^{2\overline{v}}-1-2\overline{v}).$$
First, we note
$$e^{2\overline{v}}\ge 1+2\overline{v}.$$
Then,
$$\textit{S}(\overline{v})
\le d\Delta \overline{v}-\Delta d-\frac2d\overline{v}.$$
Hence,
\begin{align*} \textit{S}(\overline{v})&\le \frac12\kappa(0)d\Delta d+Ad\Delta w+Bd\Delta d^{1+\alpha}\\
&\qquad-\Delta d
-\kappa(0)-2A(x'^2+d^2)^{\frac\alpha2}-2Bd^\alpha.\end{align*}
Straightforward calculations yield
$$|d \Delta w| \leq C(d^{\alpha}+w),$$
where $C$ is a positive constant depending only on the geometry of  $\Omega$ near $x_0$.
Note
$$|\Delta d+\kappa(0)|\leq K(|x'|^2+d^2)^{\frac\alpha2},$$
for some positive constant $K $ depending only on the geometry of  $\Omega$ near $x_0$.
Then,
\begin{align*} \textit{S}( \overline{v}) &\leq  CAd^{\alpha} +B[\alpha(\alpha+1)+(1+\alpha)dI_d-2]d^{\alpha}\\
&\qquad+(CA d-2A)(x'^{2}+d^2)^{\frac{\alpha}{2}}+K(|x'|^2+d^2)^{\frac\alpha2} +C d.\end{align*}
Since $\alpha<1$, we can take $r$ sufficiently small such that
$$2-\alpha(\alpha+1)-(1+\alpha)dI_d\ge c_0\quad\text{in } G_r,$$
for some positive constant $c_0$.
By taking $r$ small further and choosing $A\ge K+C$, we have
$$ \textit{S}( \overline{v}) \leq  CAd^{\alpha} -c_0Bd^{\alpha}
\quad\text{in }G_r.$$
We take $A$ large further such that
\begin{equation*}
C_0d \leq \frac{1}{2}\kappa(0)d+ Ad(x'^{2}+d^2)^{\frac{\alpha}{2}}+Bd^{1+\alpha} \quad\text{on } \partial G_r.
\end{equation*}
Then, we take $B$ large such that
$$c_0B\ge CA.$$
Therefore,
\begin{align*}
\textit{S}( \overline{v}) &\leq \textit{S}(v)  \quad \text{in }G_r, \\
 v &\leq \overline{v} \quad   \text{on }\partial G_r.
\end{align*}
By the maximum principle, we have $ v \leq \overline{v}$ in $G_r$.

{\it Step 3.} We now construct subsolutions and prove a lower bound of $v$.
By taking the same $w$ as in \eqref{eq-definition_w} and setting,
for some positive constants $A$ an $B$ to be determined,
$$\underline{v}=\frac{1}{2}\kappa(0)d-Aw-Bd^{1+\alpha}.$$
We first assume
\begin{equation}\label{subsl-bound1}
|\kappa(0)|r +A 2^{\frac{\alpha}{2}}r^{1+\alpha}+Br^{1
+\alpha} \le \frac{2-\alpha(\alpha+1)}{16}.
\end{equation}
Then,
$$ \left|\frac{1}{d}(e^{2\underline{v}}-1
-2\underline{v})\right|\\
\le 2\kappa^2(0)d + \frac{1}{2}[2-\alpha(\alpha+1)][A(x'^{2}+d^2)^{\frac{\alpha}{2}}+Bd^{\alpha}].
$$
Arguing as in Step 2, we obtain
\begin{align*} \textit{S}( \underline{v})
&\geq  -CA d^{\alpha} +B\left[1-\frac{1}{2}\alpha(\alpha+1)-(1+\alpha)dI_d\right]d^{\alpha}\\
&\quad+(A-CA d)(x'^{2}+d^2)^{\frac{\alpha}{2}}-K(|x'|^2+d^2)^{\frac\alpha2} - C d.\end{align*}
We require
\begin{equation}\label{eq-require_r}
d\le\frac{1}{2C}, \quad 1-\frac{1}{2}\alpha(\alpha+1)-(1+\alpha)dI_d\ge c_0\quad\text{in }G_r,\end{equation}
for some positive constant $c_0$. If $A\geq 2K+2C$, we have
$$ \textit{S}( \underline{v}) \geq  -CAd^{\alpha} +c_0Bd^{\alpha}. $$
If
$$c_0B\ge CA,$$
we have $\textit{S}( \underline{v}) \geq 0$.
In order to have $ v \geq \underline{v}$ on $\partial G_r$, it is sufficient to
require
$$|\kappa(0)|+C_0  \leq Ar^{\alpha}.$$
In summary, we first choose
$$A=\frac{|\kappa(0)|+C_0 }{r^{\alpha}},\quad B=\frac{AC}{c_0},$$
for some $r$ small to be determined. Then, we choose $r$ small satisfying \eqref{eq-require_r} such that
$A\geq 2K+2C$ and \eqref{subsl-bound1} holds.
Therefore, we have
\begin{align*}
\textit{S}( \underline{v}) &\geq \textit{S}(v)  \quad \text{in }G_r ,\\
 v &\geq \underline{v}  \quad  \text{on }\partial G_r.
\end{align*}
By the maximum principle, we have $ v \geq \underline{v}$ in $G_r$.

{\it Step 4.} Therefore, we obtain
$$\underline{v}\le v\le \overline{v}\quad\text{in }G_r.$$
By taking $x'=0$, we obtain, for any $d \in (0, r)$,
$$\left|v(0,d)-\frac{1}{2}\kappa(0) d\right|\leq Cd^{1+\alpha}.$$
This is the desired estimate.
\end{proof}

We point out that the proof above can be adapted to yield a similar result as in
Theorem \ref{thrm-C-2,alpha-expansion} for the equation \eqref{eq-MainEq-HigherDim}.

\section{Expansions near Isolated Singular Boundary Points}\label{sec-IsolatedSingular}

In this section, we study
asymptotic behaviors of $u$ near isolated singular boundary points
and aim to derive optimal estimates concerning leading terms. We will prove Theorem \ref{thrm-Main}
by a combination of conformal transforms and the maximum principle.

Throughout this section,  we will adopt notations from complex analysis
and denote by $z=(x,y)$ points in the plane.

We fix a boundary point; in the following, we always
assume this is the origin.
We assume $\partial\Omega$ in a neighborhood of the origin
consists of two $C^{2}$ curves $\sigma_1$ and
$\sigma_2$. Here, the origin is
an end of both $\sigma_1$ and $\sigma_2$.
Suppose $l_1$ and $l_2$ are two rays from the origin such that
$\sigma_1$ and $\sigma_2$ are tangent to $l_1$ and $l_2$ at the origin, respectively.
The rays $l_1$ and $l_2$ divide $\mathbb R^2$ into two cones and one of the cones
is naturally defined as the tangent cone of $\Omega$ at the origin. By a rotation, we assume
the tangent cone $V_\mu$ is given by, for some positive constant $\mu\in (0,2)$,
\begin{equation}\label{eq-Cone}
V_\mu=\{(r, \theta)\in\mathbb R^2:\, 0<r<\infty, \, 0<\theta<\mu \pi\}.\end{equation}
Here, we used the polar coordinates in $\mathbb R^2$. In fact, the tangent cone $V_\mu$
can be characterized by the following: For any $\varepsilon>0$, there exists an $r_0>0$ such that
$$\{(r,\theta): r\in(0,r_0), \theta\in(\varepsilon, \mu\pi-\varepsilon)\}
\subset \Omega\cap B_{r_0}\subset
\{(r,\theta): r\in (0,r_0), \theta\in (-\varepsilon, \mu\pi+\varepsilon)\}.$$

Our goal is to approximate solutions near an isolated singular boundary point by
the corresponding solutions in tangent cones. To this end, we express explicitly the solutions
in tangent cones. For any constant $\mu\in (0,2)$, consider the unbounded cone
$V_{\mu}$ defined by \eqref{eq-Cone}.
Then, the solution of
\eqref{eq-MainEq}-\eqref{eq-MainBoundary} in $V_\mu$ is given by
\begin{equation}\label{eq-Solution-Cone}
v_\mu= -\log\left(\mu  r  \sin\frac{\theta}{\mu}\right).\end{equation}
For $\mu\in (0,1)$ and $\theta\in (0, \mu\pi/2)$, we have $d=r\sin\theta$ and
\begin{equation}\label{eq-Solution-Cone1}v_\mu=-\log d-\log\frac{\mu\sin\frac{\theta}\mu}{\sin\theta}.
\end{equation}
For $\mu\in (1,2)$, if $\theta\in (0, \pi/2)$, we have $d=r\sin\theta$ and the identity above;
if $\theta\in (\pi/2, \mu\pi/2)$, we have $d=r$ and
\begin{equation}\label{eq-Solution-Cone2}v_\mu=-\log d-\log\left(\mu\sin\frac{\theta}{\mu}\right).
\end{equation}
We note that the second terms in
\eqref{eq-Solution-Cone1} and \eqref{eq-Solution-Cone2} are constant along the ray from the origin.
This suggests that Lemma \ref{lemma-ExteriorCone}
cannot be improved in general if the boundary has a singularity.

Next, we modify the solution in \eqref{eq-Solution-Cone} and construct super- and subsolutions. Define
\begin{equation}\label{eq-superSolution-Cone}
\overline{u}_\mu=v_\mu+\log  \left(1+A|z|^{\frac{\sqrt{2}}{\mu}}\right),\end{equation}
and
\begin{equation}\label{eq-subSolution-Cone}\underline{u}_\mu=v_\mu-\log  \left(1+A|z|^{\frac{1}{\mu}}\right),
\end{equation}
where $v_\mu$ is given by \eqref{eq-Solution-Cone} and $A$ is a positive constant.

\begin{lemma}\label{lemma-Super-sub-solutions}
Let $V_{\mu}$ be the cone defined in \eqref{eq-Cone},
and $\overline{u}_\mu$ and $\underline{u}_\mu$ be defined by
\eqref{eq-superSolution-Cone} and \eqref{eq-subSolution-Cone}, respectively.
Then, $\overline{u}_\mu$ is a supersolution
and $\underline{u}_\mu$ is a subsolution of
\eqref{eq-MainEq}  in $V_\mu$, respectively.
\end{lemma}

\begin{proof} We calculate in polar coordinates. For functions of $r$ only, we have
$$\Delta=\partial_{rr}+\frac1r\partial_r.$$
Note $r=|z|$. A straightforward calculation yields
$$\Delta\left(\log  \left(1+A|z|^{\frac{\sqrt{2}}{\mu}}\right)\right)
=\frac{2}{\mu^2r^2} \cdot \frac{Ar^{\frac{\sqrt{2}}{\mu}} }{1+Ar^{\frac{\sqrt{2}}{\mu}}}
- \frac{2}{\mu^2r^2} \left(\frac{Ar^{\frac{\sqrt{2}}{\mu}} }{1+Ar^{\frac{\sqrt{2}}{\mu}}} \right)^2.$$
Then,
\begin{align*}
\Delta \overline{u}_\mu
&=\frac{1}{\mu^2r^2\sin^2\frac\theta\mu}+\frac{2}{\mu^2r^2}
\cdot \frac{Ar^{\frac{\sqrt{2}}{\mu}} }{1+Ar^{\frac{\sqrt{2}}{\mu}}}
- \frac{2}{\mu^2r^2} \left(\frac{Ar^{\frac{\sqrt{2}}{\mu}} }{1+Ar^{\frac{\sqrt{2}}{\mu}}} \right)^2\\
&=\frac{1}{\mu^2r^2\sin^2\frac\theta\mu}
\left(1+\frac{2Ar^{\frac{\sqrt{2}}{\mu}} }{1+Ar^{\frac{\sqrt{2}}{\mu}}} \sin^2\frac\theta\mu
- 2\left(\frac{Ar^{\frac{\sqrt{2}}{\mu}} }{1+Ar^{\frac{\sqrt{2}}{\mu}}} \right)^2 \sin^2\frac\theta\mu\right)\\
&\le\frac{1}{\mu^2r^2\sin^2\frac\theta\mu}\left(1+2Ar^{\frac{\sqrt{2}}{\mu}}
\right)
\le \left(\frac{1}{\mu r\sin\frac\theta\mu}\right)^2\left(1+Ar^{\frac{\sqrt{2}}{\mu}} \right)^2=e^{2\overline{u}_\mu}.
\end{align*}
Hence, $\overline{u}_\mu$ is a supersolution in $V_\mu$.

The proof for $\underline{u}_\mu$ is similar. In fact, we have
\begin{align*}
\Delta \underline{u}_\mu
&=\frac{1}{\mu^2r^2\sin^2\frac\theta\mu}
-\frac{1}{\mu^2r^2} \cdot \frac{Ar^{\frac{1}{\mu}} }{1+Ar^{\frac{1}{\mu}}}
+ \frac{1}{\mu^2r^2} \left(\frac{Ar^{\frac{1}{\mu}} }{1+Ar^{\frac{1}{\mu}}} \right)^2\\
&=\frac{1}{\mu^2r^2\sin^2\frac\theta\mu}
\left(1-\frac{Ar^{\frac{1}{\mu}} }{1+Ar^{\frac{1}{\mu}}} \sin^2\frac\theta\mu
+\left(\frac{Ar^{\frac{1}{\mu}} }{1+Ar^{\frac{1}{\mu}}} \right)^2 \sin^2\frac\theta\mu\right)\\
& \ge\frac{1}{\mu^2r^2\sin^2\frac\theta\mu}\left(1-\frac{Ar^{\frac{1}{\mu}} }{1+Ar^{\frac{1}{\mu}}} \right)
\ge
\left(\frac{1}{\mu r\sin\frac\theta\mu}\right)^2\left(1+Ar^{\frac{1}{\mu}} \right)^{-2}=e^{2\underline{u}_\mu}.
\end{align*}
Hence, $\underline{u}_\mu$ is a subsolution in $V_\mu$.
\end{proof}

Next,  we quote a classical formula describing how solutions of \eqref{eq-MainEq}
change  under one-to-one holomorphic mappings. See \cite{Bandle&Flucher1996}.


\begin{lemma}\label{lemma-solution_under_conformal_trans}
Let $ \Omega_1$ and $\Omega_2$ be two domains in $\mathbb R^2$.
Suppose $u_2\in C^{2}(\Omega_2)$ is a solution of
\eqref{eq-MainEq}
in $\Omega_2$
and $f$ is  a one-to-one holomorphic function
from $\Omega_1$ onto $\Omega_2$. Then,
$$u_1(z)=u_2(f(z))+\log|f'(z)|$$ is a solution of
\eqref{eq-MainEq}
in $\Omega_1$.
\end{lemma}

\begin{proof}
Note that $g_2=e^{ 2u_2 }(dx\otimes dx+dy\otimes dy)$ is
a complete metric with constant Gauss curvature $-1$ on $\Omega_{2}$.
Since the Gauss curvature of the pull-back metric remains the same
under the conformal mapping,
then $g_1=f^{*}g_2=e^{ 2u_1 }(dx\otimes dx+dy\otimes dy)$ is
a complete metric with constant Gauss curvature $-1$ on $\Omega_{1}$.
Hence, $u_1$ solves \eqref{eq-MainEq}
in $\Omega_1$.
\end{proof}

Next, we prove that asymptotic expansions near singular boundary points up to certain orders
are local properties.

\begin{lemma}\label{lemma-Localization}
Let  $\Omega_{1}$ and $\Omega_{2}$ be two domains
which coincide in $B_{r_0}$, for some $r_0>0$, and let
$\partial\Omega_1\cap B_{r_0}$ consist of two $C^{2}$-curves
$\sigma_1$ and $\sigma_2$ intersecting at the origin with an angle $\mu\pi$, for some
constant $\mu\in (0,2)$. Suppose $V_\mu$ is the
tangent cone of $\Omega_1$ and $\Omega_2$ at the origin,
and that $u_1$ and $u_2$
are the $C^2$-solutions of \eqref{eq-MainEq}-\eqref{eq-MainBoundary}
in $\Omega_{1}$ and $\Omega_{2}$, respectively. Then,
\begin{equation}\label{sol-loc}
|u_1-u_2|\le C|z|^{\frac{1}{\mu}}\quad\text{in }\Omega_1\cap B_r,
\end{equation}
where $r$ and $C$  are positive constants depending only on $r_0$,  $\mu$ and the
$C^2$-norms of $\sigma_1$ and $\sigma_2$.
\end{lemma}

\begin{proof}
Taking $\widetilde{\mu}$ such that $\mu<\widetilde{\mu}<\min \{\sqrt{2}\mu ,2\}$ and set
\begin{equation}
\widetilde{V}_{\widetilde{\mu}}=\left\{(r, \theta)\in\mathbb R^2:\, 0<r<\infty,
\, -\frac{\widetilde{\mu}-\mu}{2}\pi<\theta<\frac{\widetilde{\mu}+\mu}{2}\pi\right\}.\end{equation}
For some constant $\delta_{1}>0,$
we have
$$\Omega_1\cap B_{\delta_1}\subseteq \widetilde{V}_{\widetilde{\mu}}.$$
Set
$$\widetilde \theta=\theta+\frac12(\widetilde \mu-\mu)\pi.$$
By Lemma \ref{lemma-ExteriorCone}, we have, for $A_1$ sufficiently large,
\begin{equation*}
u_1(z)\ge-\log\left(\widetilde{\mu } |z|  \sin\frac{\widetilde{\theta}}{\widetilde{\mu}}\right)
-\log  \left(1+A_1|z|^{\frac{1}{\widetilde{\mu}}}\right)
\quad\text{on } \Omega_1\cap \partial B_{\delta_1}.\end{equation*}
The estimate above obviously holds on $\partial\Omega_1\cap B_{\delta_1}$.
By Lemma \ref{lemma-Super-sub-solutions} and the maximum principle, we have
\begin{equation}\label{eq-subsolution-3}
u_1(z)\ge-\log\left(\widetilde{\mu } |z| \sin\frac{\widetilde{\theta}}{\widetilde{\mu}}\right)
-\log  \left(1+A_1|z|^{\frac{1}{\widetilde{\mu}}}\right)
\quad\text{in } \Omega_1\cap B_{\delta_1}. \end{equation}
In particular, we can take $\delta_2<\delta_1$ such that
\begin{equation*}
e^{2u_1}\ge\frac{1}{2\mu^{2}|z|^2}\quad\text{in } \Omega_1\cap B_{\delta_2}. \end{equation*}
As in the proof of Lemma \ref{lemma-Super-sub-solutions},
we can verify that $u_1-\log  \left(1+A|z|^{\frac{1}{\mu}}\right)$ is a subsolution of
\eqref{eq-MainEq}  in $\Omega_1\bigcap B_{\delta_2}$.
By Lemma \ref{lemma-ExteriorCone} and the maximum principle, we have,
for $A$ sufficiently large,
\begin{equation*}
u_1\leq u_2 + \log  \left(1+A|z|^{\frac{1}{\mu}}\right)\quad\text{in } \Omega_1\cap B_{\delta_2}.
\end{equation*}
Similarly, we have
\begin{equation*}
u_2\leq u_1 + \log  \left(1+A|z|^{\frac{1}{\mu}}\right)\quad\text{in } \Omega_1\cap B_{\delta_2}.
\end{equation*}
This implies the desired result.
\end{proof}

Now we prove a simple calculus result.

\begin{lemma}\label{curve-distance-x-x'-1}
Let $\sigma$ be a curve defined by a function $y=\varphi(x)\in C^{1,\alpha}([0, \delta])$, for some
constants $\alpha\in (0,1]$ and $\delta> 0$,  satisfying $\varphi(0)=0$ and  \begin{equation*}
|\varphi' (x)| \leq Mx^{\alpha}, \end{equation*}
for some positive constant $M$.
For any given point $z=(x,y)$ with $0 < x <\delta$ and $y>\varphi(x)$, let $p=(x', \varphi(x'))$
be any closest point to $z$ on $\sigma$ with the distance $d$.
Then, for $|z|$ sufficient small,
$$x' \le 2|z|.$$
Moreover, if $|y| \leq x/4 $, then
$$|x-x'|\le C d x^{\alpha},$$
where $C$ is a positive constant depending only on $M$ and $\alpha$.
\end{lemma}

\begin{proof} First, we note $d\le |z|$ since $d$ is the distance of $z$ to $\sigma$. Then,
$$x'\le |p|\le |z|+ |z-p|=|z|+d\le 2|z|.$$
Next, for $x'\in (0,\delta)$,  $x'$ is characterized by
\begin{equation*}
\frac{d}{dt}[(x-t)^2  + (y-  \varphi(t))^2 ]|_{t=x'}=0,
\end{equation*}
or
\begin{equation*}
x-x'= (y-  \varphi(x'))\varphi' (x').
\end{equation*}
If $|y| \leq x/4 $, then $|z|\le 5x/4$ and hence $x'\le 5x/2$. Moreover,
$|y-\varphi(x')|\le d$. Then,
\begin{equation}\label{curve -distance-x-x'-2}
|x-x'| \leq d | \varphi' (x')|.
\end{equation}
This implies the desired result.
\end{proof}

We are ready to discuss the case when the opening angle of the tangent cone of $\Omega$
at the origin is less than $\pi$. We first introduce the leading term.
Let $\partial\Omega$ in a neighborhood of the origin consist of two $C^{2}$-curves
$\sigma_1$ and $\sigma_2$ intersecting at the origin at an angle $\mu\pi$, for some
constant $\mu\in (0,1]$. Define,
for any $z\in \Omega$,
\begin{equation}\label{eq-definition_f0}
f_\mu(z)=- \log \left(\mu |z| \sin \frac{\arcsin \frac{d(z)}{|z|}}{\mu} \right),\end{equation}
where $d$ is the distance to
$\partial\Omega$.
We can also write, for $z$ sufficiently small,
$$
f_{\mu}(z)=
\begin{cases}
-\log (\mu |z|  \sin\frac{\arcsin\frac{d_{1}(z)}{|z|}}{\mu})
& \text{if } d_1(z) \le d_2(z),\\
-\log(\mu |z| \sin\frac{\arcsin\frac{d_{2}(z)}{|z|}}{\mu})
& \text{if }d_1(z)>d_2(z),\\
\end{cases}
$$
where $d_1$ and $ d_2$ are the distances to $\sigma_1$ and $ \sigma_2$, respectively. 
We note that $f_{\mu}=-\log d$ if $\mu=1$.

\begin{theorem}\label{thrm-SmallAngles}
Let $\Omega$ be a bounded domain in $\mathbb R^2$ and
$\partial\Omega\cap B_{r_0}$ consist of two $C^{2}$-curves
$\sigma_1$ and $\sigma_2$ intersecting at the origin with an angle $\mu\pi$, for some
constants $\mu\in (0,1]$ and $r_0>0$. Suppose $ u \in C^{2}(\Omega)$ is  a solution of
\eqref{eq-MainEq}-\eqref{eq-MainBoundary}.
Then,
for any $z\in \Omega\cap B_\delta$,
\begin{equation}\label{eq-main_estimate_small}
\left|u(z)-f_\mu(z)\right| \leq Cd(z),\end{equation}
where $f_\mu$ is given by
\eqref{eq-definition_f0},
$d$ is the distance to
$\partial\Omega$,
and $\delta$ and $C$ are positive constants depending only on
$\mu$, $r_0$ and the $C^2$-norms of $\sigma_1$ and $\sigma_2$.
\end{theorem}

We first describe the proof. Our goal is to approximate the solution $u$ in $\Omega$ by
the corresponding solution $v$ in the tangent cone $V$ of $\Omega$ at the origin in terms of the
distance $d$ to $\partial\Omega$.
Take a conformal transform $T$, with $T(0)=0$, such that
$\widetilde \Omega=T(\Omega)$ has a $C^{1,\mu}$-boundary near the origin. Then, the tangent cone
$\widetilde V$ of $\widetilde \Omega$ at the origin is a half-plane.
We can approximate the solution $\widetilde u$
in $\widetilde\Omega$ with the corresponding solution $\widetilde v$ in the tangent cone $\widetilde V$
in terms of the distance $\widetilde d$ to $\partial\widetilde \Omega$. To transform such an approximation of
$\widetilde u$ by $\widetilde v$ to that of $u$ by $v$, we need to discuss the relation between $\widetilde d$
and $d$. We are able to establish an optimal relation when points are relatively away from the boundary.
We consider two cases by different methods: points away from the boundary by a curve
transversal to the boundary at the origin (Case 1 in the proof below)
and points bounded by the above curve and another curve
tangent to the boundary at the origin up to degree 2 (Case 2 below). For these two cases,
the conformal transform $T$ plays an important role. For the rest of the points
(Case 3 below), we construct appropriate functions and
compare $u$ and $v$ directly by the maximum
principle.

Throughout the proof, we need to estimate various geometric quantities, such as
distances to curves and angles between two straight lines.
Many of these estimates are trivial if the boundary $\sigma_i$
is a line and,  these estimates follow from
approximations if $\sigma_i$ is a $C^2$-curve.

\begin{proof}
We first consider the case  $\mu=1$. In this case, 
$\sigma_1 \bigcup\sigma_2$ is a $C^{1,1}$-curve near the origin, since $\sigma_1$ and $\sigma_2$ 
are $C^2$ up to the origin and form 
an angle $\pi$ at the origin. 
Then, the conclusion follows from Theorem
3.1 for  $\alpha=1$.

We now consider the case $\mu<1$. We denote by $d_1$ and $d_2$ the distances to
$\sigma_1$ and $ \sigma_2$, respectively. We only consider the case  $d_{1}=d\leq d_2$.
We also denote by $M$ the $C^{2}$-norm of $\sigma_1$ and $\sigma_2$.
We will prove \eqref{eq-main_estimate_small} with $d=d_1$.
In the following, $C$ and $\delta$ are positive constants depending only
on $\mu$, $r_0$ and the $C^2$-norms of $\sigma_1$ and $\sigma_2$.

By restricting to a small neighborhood
of the origin, we assume $\sigma_1$ and $\sigma_2$ are curves over their tangent lines
at the origin. We also assume $\sigma_1$ is given by the function $y=\varphi_1(x)$ satisfying
$ \varphi_1(0) =0$,  $\varphi_1'(0) =0$ and
$$ |\varphi_1'' (x)| \leq M.$$
Consider the conformal homeomorphism $T: z\mapsto z^{\frac{1}{\mu}}$.
For
\begin{equation}\label{eq-coordinates}z=(x,y)=(|z| \cos \theta, |z| \sin \theta),\end{equation} we write
\begin{equation}\label{eq-tilde-coordinates}T(z)=\widetilde{z}=(\widetilde{x}, \widetilde y)
=\left(|z|^{\frac{1}{\mu}} \cos \frac{\theta}{\mu}, |z|^{\frac{1}{\mu}}  \sin \frac{\theta}{\mu}\right).\end{equation}
Set $\widetilde \sigma_i=T(\sigma_i)$, $i=1, 2$, and
$\widetilde \sigma=\widetilde\sigma_1\cup \widetilde\sigma_2$. We first
study the regularity of $\widetilde\sigma$. By expressing $\widetilde \sigma$ by
$\widetilde y=\widetilde\varphi(\widetilde x)$, we claim
\begin{equation}\label{eq-estimates_on_new_curve}
|\widetilde{\varphi} (\widetilde{x})| \leq \widetilde{M} \widetilde{x}^{1+\mu}, \quad
|\widetilde{\varphi}'(\widetilde{x})|\leq \widetilde{M}\widetilde{x}^{\mu}, \quad
|\widetilde{\varphi}''(\widetilde{x})|\leq \widetilde{M}\widetilde{x}^{\mu-1},
\end{equation}
where $\widetilde M$ is a positive constant depending only on $M$ and $\mu$.

To prove \eqref{eq-estimates_on_new_curve},
we assume $\widetilde \sigma_1=T(\sigma_1)$ is given by
$\widetilde{y} = \widetilde{\varphi}_1 (\widetilde{x})$.
To prove
the estimate of $\widetilde\varphi_1$, we note $|y|\le Cx^2$ on $\sigma_1$ and
$|z|\le C\widetilde x^\mu$ on  $\widetilde\sigma_1$ for $|z|$ sufficiently small. Then,
on  $\widetilde\sigma_1$,
$$|\widetilde y|=|z|^{\frac1\mu-1}\left||z|\sin\frac\theta \mu\right|
\le C|z|^{\frac1\mu-1}|y|\le C|z|^{\frac1\mu-1}x^2
\le C|z|^{1+\frac1\mu}\le C\widetilde x^{1+\mu}.$$
This is the first estimate in \eqref{eq-estimates_on_new_curve}.
Next, we prove estimates of derivatives of $\widetilde\varphi_1$.
By \eqref{eq-coordinates} and \eqref{eq-tilde-coordinates},
we note that $(\widetilde x, \widetilde y)$ on $\widetilde \sigma_1$ is given by
$$
\widetilde{x}=(x^{2}+\varphi_{1}(x)^{2})^{\frac{1}{2\mu}}
\cos\frac{\arcsin \frac{\varphi_{1}(x)}{((x^{2}+\varphi_{1}(x)^{2})^{\frac{1}{2}}}}{\mu},$$
and
$$\widetilde{y}=(x^{2}+\varphi_{1}(x)^{2})^{\frac{1}{2\mu}}
\sin\frac{\arcsin \frac{\varphi_{1}(x)}{((x^{2}+\varphi_{1}(x)^{2})^{\frac{1}{2}}}}{\mu}.$$
Straightforward calculations yield
\begin{align*}
\left|\frac{d\widetilde{x}}{d x}-\frac{1}{\mu}x^{\frac{1}{\mu}-1}\right|&\le Cx^{\frac{1}{\mu}},\\
\left|\frac{d^{2}\widetilde{x}}{d x^{2}}-\frac{1}{\mu}\left(\frac{1}{\mu}-1\right)
x^{\frac{1}{\mu}-2}\right|&\le Cx^{\frac{1}{\mu}-1},\end{align*}
and
\begin{align*}
\left|\frac{d\widetilde{y}}{d x}\right| &\leq \frac{1}{\mu}
\left(\frac{1}{\mu}+1\right)M x^{\frac{1}{\mu}}(1+Cx),\\
\left|\frac{d^{2}\widetilde{y}}{d x^{2}}\right|
&\leq \frac{1}{\mu^{2}}\left(\frac{1}{\mu}+1\right)M x^{\frac{1}{\mu}-1}(1+Cx).\end{align*}
With $x\le C\widetilde{x}^{\mu}$, we get the second and third estimates in
\eqref{eq-estimates_on_new_curve}. This finishes the proof of \eqref{eq-estimates_on_new_curve}
for $\widetilde x\ge 0$. A similar argument holds for $\widetilde x<0$.

We now discuss three cases for $z\in \Omega\cap B_\delta$
with $d_1(z)\le d_2(z)$, for $\delta$ sufficiently small.
We set
\begin{align}\label{eq-Definition_Omega} \begin{split}
\Omega_1&=\{z\in\Omega:\, d_1(z)> c_0|z|\},\\
\Omega_2&=\{z\in\Omega:\, c_1|z|^2<d_1(z)< c_0|z|\},\\
\Omega_3&=\{z\in\Omega:\, d_1(z)< c_1|z|^2\},\end{split}\end{align}
and \begin{align}\label{eq-Definition_gamma}\begin{split}
\gamma_1&=\{z\in\Omega:\, d_1(z)= c_0|z|\},\\
\gamma_2&=\{z\in\Omega:\, d_1(z)= c_1|z|^2\},\end{split}\end{align}
where $c_0$ and $c_1$ are appropriately chosen constants
with $c_0 <\frac{1}{2} \mu \arctan \frac{1}{4}$.
We point out that
$\gamma_1$ is transversal to $\sigma_1$, or the positive $x$-axis, at the origin, and that
$\gamma_2$ is tangent to $\sigma_1$ at the origin.

We will prove \eqref{eq-main_estimate_small} in $\Omega_1$, $\Omega_2$,
and $\Omega_3$ by considering these three cases separately.
In the first case, we prove \eqref{eq-main_estimate_small} in $\Omega_1$, for any
$c_0>0$. In the second case, we prove \eqref{eq-main_estimate_small} in $\Omega_2$, for any
$c_0>0$ sufficiently small and any $c_1>0$. We fix $c_0$ in this case.
In the third case, we prove \eqref{eq-main_estimate_small} in $\Omega_3$, for an
appropriate $c_1>0$.

Let $V$ be the tangent cone of $\Omega$ at the origin given by
\eqref{eq-Cone}
and $v$ be the solution of \eqref{eq-MainEq}-\eqref{eq-MainBoundary}
in  $V$ given by
\eqref{eq-Solution-Cone}.

{\it Case 1.} We consider $z\in \Omega_1\cap B_\delta$.

Set
$$\Omega_{+}= \Omega\cap B_{\delta}, \quad
\Omega_{-}=\Omega \cup B_{\delta}^{c}.$$
Let $u_+$  and $u_-$
be the solutions of \eqref{eq-MainEq}-\eqref{eq-MainBoundary}
in $\Omega_{+}$ and $\Omega_{-}$, respectively.  By the maximum principle, we have
\begin{equation}\label{case1-1}u_- \leq u \leq u_+ \quad \textrm{in }\Omega_+.\end{equation}
We take $\delta$ small so that $T$ is one-to-one on $\Omega_+$.
We point out that $\Omega_+$ is a bounded domain in $\Omega$
and that $\Omega_-$ is an unbounded domain containing $\Omega$.
In the following, we compare $u_+$ and $u_-$ with $v$
and establish \eqref{case1-2} and \eqref{case1-3}.

First, we compare $u_+$ with $v$. Set
$\widetilde{\Omega}_{+}=T(\Omega_+)$ and let $\widetilde u_+$ be the solution of
\eqref{eq-MainEq}-\eqref{eq-MainBoundary} in $\widetilde\Omega_+$.
We will compare $\widetilde u_+$ in $\widetilde{\Omega}_{+}$ with the corresponding
solution $\widetilde v$ in the tangent cone $\widetilde V$ of $\widetilde{\Omega}_{+}$ at the origin
and then use the relations between $u_+$ and $\widetilde u_+$ and between
$v$ and $\widetilde v$ to get the desired estimate concerning $u_+$ and $v$.
By \eqref{eq-estimates_on_new_curve}, the curve $\widetilde\sigma$ given by
$\widetilde y=\widetilde \varphi(\widetilde x)$
satisfies
$$-\widetilde M|\widetilde x|^{1+\mu}\le |\widetilde\varphi(\widetilde x)|\le \widetilde M|\widetilde x|^{1+\mu}.$$
Theorem \ref{thrm-C-1,alpha-expansion} implies, for $\widetilde z$ close to the origin,
$$|\widetilde{u}_+(\widetilde{z})+\log\widetilde{ d}|\le C\widetilde{d}^{\mu},$$
where $\widetilde d$ is the distance from $\widetilde z$ to the curve $\widetilde \sigma$.  Therefore,
for $\widetilde z$ close to the origin,
\begin{equation}\label{case1-lower}
\widetilde{u}_+(\widetilde{z})\le -\log\widetilde{ d}_{1} +C\widetilde{d}_{1}^{\mu},
\end{equation}
and
\begin{equation}\label{case1-upper}
\widetilde{u}_+(\widetilde{z})\geq -\log \widetilde{ d}_{2} -C\widetilde{d}_{2}^{\mu},
\end{equation}
where $\widetilde{d}_1$ and $\widetilde{d}_2$ are the distances from $\widetilde{z}$ to the curves
$\widetilde{y}=\widetilde{M}|\widetilde{x}|^{1+\mu}$
and
$\widetilde{y}=-\widetilde{M}|\widetilde{x}|^{1+\mu}$, respectively.
Let $(\widetilde x', \widetilde y')$ be the point on $\widetilde{y}=\widetilde{M}|\widetilde{x}|^{1+\mu} $
realizing the distance from $\widetilde z=(\widetilde x, \widetilde y)$. Then,
$$\widetilde y-\widetilde y'\le \widetilde d_1\le \widetilde y,$$
and hence
$$|\widetilde d_1-\widetilde y|\le \widetilde y'=\widetilde\varphi(\widetilde x')\le
\widetilde M|\widetilde x'|^{1+\mu}
\le C\left(|z|^{\frac1\mu}\right)^{1+\mu}=C|z|^{1+\frac1\mu}.$$
Recall $d_1\ge c_0|z|$  since $z\in \Omega_1\cap B_\delta$. By a simple
geometric argument,
we have $\theta\ge \theta_0$ for some positive constant $\theta_0$,
for $|z|$ sufficiently small. Here, $\theta$ is the angle between $Oz$ and the positive
$x$-axis, as in \eqref{eq-coordinates}. As a consequence, we get
$|z|\le Cy\le C\widetilde y^\mu$. Hence,
$|\widetilde{d}_1-\widetilde{y}|\le C\widetilde y|z|$, or
$$\left|\frac{\widetilde{d}_1}{\widetilde{y}}-1\right|\le C|z|.$$
Therefore,
$$|\log \widetilde d_1-\log  \widetilde y|\le C|z|\le Cd_1.$$
Next, we note
$$\widetilde d_1^\mu\le |\widetilde z|^\mu=|z|\le Cd_1.$$
Similar estimates hold for $\widetilde d_2$.
Then,
\eqref{case1-lower} and \eqref{case1-upper}
imply
\begin{equation}\label{eq-estimate_u_+}
|\widetilde u_+(\widetilde z)+ \log  \widetilde y|\le Cd_1.\end{equation}
Recall that $V$ is the tangent cone of $\Omega$ at the origin given by
\eqref{eq-Cone}
and $v$ is the solution of \eqref{eq-MainEq}-\eqref{eq-MainBoundary}
in  $V$ given by
\eqref{eq-Solution-Cone}. Then, $T(V)$ is the upper half-plane and the solution
$\widetilde v$ of \eqref{eq-MainEq}-\eqref{eq-MainBoundary} in $T(V)$
is given by
\begin{equation*}
\widetilde{v}(\widetilde{z})= -\log \widetilde{ y}.
\end{equation*}
Hence, \eqref{eq-estimate_u_+} implies
$$|\widetilde u_+(\widetilde z)- \widetilde v(\widetilde z)|\le Cd_1.$$
By Lemma \ref{lemma-solution_under_conformal_trans}, we have
\begin{equation*}
u_+(z)=\widetilde{u}_+(\widetilde{z})+ \log \left(\frac{1}{\mu}|z|^{\frac{1}{\mu}-1}\right),
\end{equation*}
and
\begin{equation*}
v(z)=\widetilde{v}(\widetilde{z})+ \log\left(\frac{1}{\mu}|z|^{\frac{1}{\mu}-1}\right).
\end{equation*}
Therefore, 
we obtain
\begin{equation}\label{case1-2}|u_+(z)-v(z)|\le Cd_1.\end{equation}

Next, we compare $u_-$ with $v$. We could use the similar method as
comparing $u^+$ and $v$ as above to achieve this. Instead, we
transform the unbounded domain $\Omega_-$ to a bounded domain conformally
and reduce the present situation to what we just discussed. Then, we can employ
\eqref{case1-2} directly.
To this end, we fix a point $P\in \Omega_{-}^{c}$ and
consider the conformal homeomorphism $\widehat T: z\mapsto \frac{1}{z-P}$.
We assume that $\widehat T$ maps
$\Omega_{-}$ to $\widehat{\Omega}_{-}$, $V$ to $\widehat{V}$,
$\sigma_i$ to $\widehat{\sigma}_i$, and $ l_i$ to $\widehat{l}_i$.
Then, $\widehat{\sigma}_i$ and $ \widehat{l}_i$ are $C^{2}$-curves
with bounded $C^{2}$-norms in a small neighborhood of $\widehat T(0)$ since
$\widehat T$ is smooth  in
$\overline{B}_{|0P|/2}$. The tangent cone of
$\widehat{\Omega}_{-}$ at $\widehat T(0)$, denoted by $\underline{V}$,  has an opening angle
$\mu\pi$ since $\widehat T$ is conformal. Let $\widehat{u}_{-}$, $\widehat v$ and $\underline{v}$
be the solutions of \eqref{eq-MainEq}-\eqref{eq-MainBoundary}
in $\widehat{\Omega}_{-}$, $\widehat V$ and  $\underline{V}$, respectively.
By Lemma \ref{lemma-solution_under_conformal_trans}, we have
\begin{equation*}
u_{-}(z)=\widehat{u}_{-}(\widehat{z})-  2\ln|z- P|,
\end{equation*}
and
\begin{equation*}
v(z)=\widehat{v}(\widehat{z})-  2\ln|z- P|.
\end{equation*}
By applying \eqref{case1-2}, with $u_+$ in $\Omega_{+} $ replaced by
$\widehat{u}_{-}$ and $\widehat v$ in $\widehat{\Omega}_{-}$ and $\widehat V$,
respectively, and $v$ in $V$ replaced by $\underline{v}$ in $\underline{V}$, we have
$$|\widehat{u}_{-}(\widehat{z}) - \underline{v}(\widehat{z})|\le C\widehat{d},$$
and
$$|\widehat{v}(\widehat{z}) - \underline{v}(\widehat{z}) |\le C\widehat{d}.$$
We note that the distance $\widehat d$ from $\widehat z$ to $\partial\widehat\Omega_-$
is comparable to that from $\widehat z$ to $\partial\widehat V$
since $\partial\widehat\Omega_-$ and $\partial\widehat V$ are tangent at $\widehat T(0)$.
Therefore,
$$|\widehat u_{-}(\widehat z) - \widehat v(\widehat z)|\le C\widehat d,$$
and hence
\begin{equation}\label{case1-3}|u_{-}(z) - v(z) |\le Cd_1.\end{equation}

By combining \eqref{case1-1}, \eqref{case1-2} and
\eqref{case1-3}, we have
$$|u(z) - v(z) |\le Cd_1.$$
By the explicit expression of $v$ in \eqref{eq-Solution-Cone}, it is straightforward  to verify
$$\left|v(z) +\log \left(\mu |z| \sin \frac{\arcsin \frac{d_1}{|z|}}{\mu} \right) \right|\le Cd_1. $$
We hence have \eqref{eq-main_estimate_small}  for $z\in \Omega_1\cap B_\delta$.

{\it Case 2.} We consider $z\in \Omega_2\cap B_\delta$ and discuss in  two cases.

{\it Case 2.1.} First, we assume $T$ is one-to-one in $\Omega$. Set
$\widetilde \Omega=T(\Omega)$ and let $\widetilde u$
be the solution of \eqref{eq-MainEq}-\eqref{eq-MainBoundary}
in $\widetilde \Omega$.
Let $\widetilde{p}=(\widetilde{x}', \widetilde{y}')$
be the closest point on $\widetilde{\sigma}_1$ to $\widetilde{z}=(\widetilde x, \widetilde y)$
with the distance $\widetilde{d}$.
We first demonstrate that we are able to
put a ball in $\widetilde\Omega$ and another ball outside $\widetilde \Omega$,
both of which are tangent to $\partial\widetilde\Omega$ at $\widetilde p$, and
then compare $\widetilde u$ with the corresponding solutions associated with these tangent balls.
Based on this, we can compare $\widetilde u$ with the solution in {\it some} half-space, which is
close to the tangent cone $\widetilde V$ of $\widetilde \Omega$ at the origin. Then, we can compare
$u$ with the solution $v'$ in some cone,
which is close to the tangent cone $V$ of $\Omega$ at the origin,
as shown in \eqref{eq-Case2.1.1}. Last, we compare $v'$ with $v$  in \eqref{eq-Case2.1.2}.
We now proceed with the proof.

For $z\in \Omega_2$, $d_1<c_0|z|$.
If $c_0$ is small, then $|\widetilde y|\le c_*|\widetilde x|$, for some constant
$c_*$ small. This follows easily from the relation between $z$ and $\widetilde z$, as given by
\eqref{eq-coordinates} and \eqref{eq-tilde-coordinates}.
By Lemma \ref{curve-distance-x-x'-1}, we have
\begin{equation*}\label{eq-diffrence}|\widetilde{x}'- \widetilde{x}| \le C \widetilde{x}'^{\mu}\widetilde{d}.
\end{equation*}
Note that $|z|$ is comparable with $x$ and that $|\widetilde z|$ is comparable with $\widetilde x$.
With $\widetilde{x}'^{\mu} \leq |z|$ and by  \eqref{eq-estimates_on_new_curve}, we have
\begin{equation}\label{eq-slope1}
|\widetilde\varphi'(\widetilde x')| \leq \widetilde M\widetilde x'^{\mu}\le C|z|,\end{equation}
and similarly
\begin{equation}\label{eq-slope2}\frac{|\widetilde\varphi(\widetilde x')|}{|\widetilde x'|}\le C|z|.\end{equation}
Next, we claim, for any $\widehat x$ sufficiently small,
\begin{equation}\label{taylor-expansion_of_curve}
|\widetilde{\varphi}(\widehat{x})-\widetilde{\varphi}(\widetilde{x}')
-\widetilde\varphi'(\widetilde x')(\widehat{x}-\widetilde{x}')|
\leq K |z|^{1-\frac{1}{\mu}}(\widehat{x} -\widetilde{x}' )^2,
\end{equation}
where $K$ is a positive constant depending only on
$M$ and $ \mu$.  We prove \eqref{taylor-expansion_of_curve} in three cases.
If $\widehat{x} \geq |z|^{\frac{1}{\mu}}/3$,
then, with $\mu\in (0,1)$,
$$ |\widetilde{\varphi} ''(\widehat{x})|\leq \widetilde{M}\widehat{x}^{\mu-1}
\le C|z|^{1-\frac{1}{\mu}},$$
and \eqref{taylor-expansion_of_curve} holds by the Taylor expansion.
If $ 0 \leq \widehat{x} \leq |z|^{\frac{1}{\mu}}/3$,  we have
\begin{align*}
&|\widetilde{\varphi}(\widehat{x})|\leq \widetilde{M}\widehat {x}^{1+\mu} \leq  C|z|^{1+\frac{1}{\mu}}, \\
&|\widetilde\varphi'(\widetilde x')(\widehat{x}-\widetilde{x}')|  \leq C|z|^{1+\frac{1}{\mu}}, \end{align*}
and
$$|z|^{1+\frac{1}{\mu}}=|z|^{1-\frac{1}{\mu}}(|z|^{\frac{1}{\mu}})^2
\leq C|z|^{1-\frac{1}{\mu}}(\widehat{x}  - \widetilde{x}' )^2 . $$
Then, \eqref{taylor-expansion_of_curve} follows.
If $ \widehat{x} \leq 0$,  we have
$$|\widetilde{\varphi}(\widehat{x})|\leq \widetilde{M}|\widehat{x}|^{1+\mu}
\leq  Cr^{1-\frac{1}{\mu}}(\widehat{x}  - \widetilde{x}' )^2, $$
and
$$|\widetilde\varphi'(\widetilde x')(\widehat{x}-\widetilde{x}')|
\leq Cr^{1-\frac{1}{\mu}}(\widehat{x}  - \widetilde{x}' )^2. $$
Then, \eqref{taylor-expansion_of_curve} also holds.

We note that the left-hand side of \eqref{taylor-expansion_of_curve} is given by
the difference of $\widetilde \varphi$ and its linear part at $\widetilde x'$. Hence, the graph of
$\widetilde\varphi$, viewed as a graph over its tangent line at $(\widetilde x', \widetilde\varphi(\widetilde x'))$,
is bounded by two parabolas.
Hence, we have, for some $R=C'|z|^{\frac{1}{\mu}-1}$,
$$B_R( \widetilde{p}+R\vec{n})\subset\widetilde\Omega
\quad\text{and}\quad
B_R( \widetilde{p}-R\vec{n})\cap\widetilde \Omega=\emptyset,$$
where $\vec{n}$ is the  unit inward normal vector of $\widetilde{\sigma}_1$ at $\widetilde{p}$
and $C'$ is some positive constant.
Let $u_{R,\widetilde{p}+R\vec{n}}$ and $v_{R, \widetilde{p}-R\vec{n}}$ be the solutions of
\eqref{eq-MainEq}-\eqref{eq-MainBoundary} in
$B_R(\widetilde{p}+R\vec{n})$ and $\mathbb R^2\setminus B_R(\widetilde{p}-R\vec{n})$,
given by \eqref{eq-solution-inside} and \eqref{eq-solution-outside},
respectively. Then, by the maximum principle, we have
$$v_{R,\widetilde{p}-R\vec{n}}\le \widetilde u\le u_{R, \widetilde{p}+R\vec{n}}
\quad\text{in }B_R(\widetilde{p}+R\vec{n}),$$
and hence, at  $\widetilde{z}$,
$$-\log \widetilde{d}-\log\left(1+\frac{\widetilde{d}}{2R}\right)\le \widetilde u
\le -\log \widetilde d-\log\left(1-\frac{\widetilde{d}}{2R}\right).$$
Therefore,
\begin{equation}\label{error_term_tilde_d}
|\widetilde{u} (\widetilde{z}) +\log \widetilde{ d} |\le \frac{C\widetilde{d}}{|z|^{\frac{1}{\mu}-1}}.
\end{equation}
For \textit{T}: $z\mapsto z^{\frac{1}{\mu}}$,
if $z_1, z_2 \in B_{|z|/3}(z)$, we have
\begin{equation*}
|\textit{T}(z_1)-\textit{T}(z_2)|\leq \frac{1}{\mu}\max_{z' \in B_{|z|/3}(z)}
\{|z'|^{\frac{1}{\mu}-1}\}|z_1 -z_2|.
\end{equation*}
Let $q $ be the closest point  on $\sigma_1$ to $z$, with the distance given by $d_1$.
By $d_1 < c_0 |z|$ for $c_0$ small,
we have $q \in B_{|z|/3}(z)$ if $|z|$ is small. Therefore,
\begin{equation}\label{eq-distance}
\widetilde{d} \leq \textrm{dist}(\widetilde{z} ,T({q})) \leq C|z|^{\frac{1}{\mu}-1} d_1.
\end{equation}
With \eqref{error_term_tilde_d}, we obtain
\begin{equation*}
|\widetilde{u} (\widetilde{z}) +\log\widetilde{ d}|\le Cd_1.
\end{equation*}
Let $\widetilde{l}$ be the tangent line of  $\widetilde{\sigma}_1$
at $\widetilde{p}$ and
$\widetilde{l}'$  be the line passing the origin and $\widetilde{p}$.
Then, the slopes of these two straight lines are bounded by $C|z|$
by \eqref{eq-slope1} and \eqref{eq-slope2}.
Therefore, the included angle $\widetilde\theta$
between  $\widetilde{l} $ and $\widetilde{l} '$  is less than $C|z|$, and hence,
\begin{equation*}
|\textrm{dist} (\widetilde{z} ,\widetilde{l}') -\widetilde d|
= |\widetilde{d}\cos\widetilde{\theta} -\widetilde d|
\le C \widetilde{d} \widetilde{\theta}^{2} \le C \widetilde{d} |z|^{2}.
\end{equation*}
By $c_1 |z|^2 \leq d_1 $, we obtain
\begin{equation*}
|\textrm{dist} (\widetilde{z} ,\widetilde{l}') -\widetilde d|\le C\widetilde{d} d_1.
\end{equation*}
Let $\widetilde{V}'$ be the half-plane above the line $\widetilde{l}'$
and $\widetilde{v}'(z)$
be the solution of \eqref{eq-MainEq}-\eqref{eq-MainBoundary}
in $\widetilde{V}'$. Then,  $\widetilde{v}'(\widetilde{z})=-\log\textrm{dist} (\widetilde{z} ,\widetilde{l}')$
and hence
$$|\widetilde{v}'(\widetilde{z}) + \log \widetilde{d}|\le C d_1. $$
Therefore,
\begin{equation}\label{expansion_case2}
|\widetilde{u}(\widetilde{z}) - \widetilde{v}'(\widetilde{z})|\le C d_1.
\end{equation}
Set $V'=\textit{T}^{-1}(\widetilde{V}')$ 
and let $v'$ be the solution of \eqref{eq-MainEq}-\eqref{eq-MainBoundary}
in $V'$. By Lemma \ref{lemma-solution_under_conformal_trans}, we get
\begin{equation*}
u(z)=\widetilde{u}(\widetilde{z})+ \log \left(\frac{1}{\mu}|z|^{\frac{1}{\mu}-1}\right),
\end{equation*}
and
\begin{equation*}
v'(z)=\widetilde{v}'(\widetilde{z})+ \log\left(\frac{1}{\mu}|z|^{\frac{1}{\mu}-1}\right).
\end{equation*}
Combining with \eqref{expansion_case2}, we have
\begin{equation}\label{eq-Case2.1.1} |u(z) -v'(z)|\le C d_1.\end{equation}
We point out that the choice of $v'$ depends on $z$.

Next, we compare $v'$ with the solution $v$ in the tangent cone $V$ of $\Omega$ at 0.
To this end, we need to compare $\mathrm{dist}(z,\partial V')$ with $d_1$, which is the distance from $z$ to $\sigma_1$.
Recall that $q$ is the closest point on $\sigma_1$ to $z$ and that
$\widetilde p$ is the closest point on $\widetilde\sigma_1$ to $\widetilde z$. Denote $p=T^{-1}\widetilde{p}$.
Set $p=(x', \varphi_1(x'))$ and $q=(\overline{x}, \varphi_1(\overline x))$. We first claim
\begin{equation}\label{eq-distances}|x'-\overline{x}|\le \frac{Cd_1^2}{|z|}+C|z|d_1.\end{equation}
To prove \eqref{eq-distances}, we will compare $x'$, $\overline{x}$ with $x$.
Since $q=(\overline{x}, \varphi_1(\overline x))$ is the closest point on $\sigma_1$
to $z=(x,y)$ with the distance $d_1$,
by Lemma \ref{curve-distance-x-x'-1} with $\alpha=1$,
we have
\begin{equation}\label{eq-distances1}|x-\overline{x}|\le Cxd_1\le C|z|d_1.\end{equation}
Since $\widetilde{p}=(\widetilde{x}', \widetilde{y}')$
is the closest point on $\widetilde{\sigma}_1$ to $\widetilde{z}=(\widetilde x, \widetilde y)$
with the distance $\widetilde{d}$,
by Lemma \ref{curve-distance-x-x'-1} again, we have
$$ |\widetilde x'- \widetilde{x}| \leq C\widetilde{x}^{\mu }\widetilde{d},$$
and hence
$$\operatorname{dist}(\widetilde{p}, (\widetilde{x}, \widetilde{\varphi}_{1}(\widetilde{x}) )\le
C\widetilde{x}^{\mu }\widetilde{d}.$$
Set
$(x_*, \varphi_{1}(x_*) )=T^{-1}(\widetilde{x}, \widetilde{\varphi}_{1}(\widetilde{x}))$.
Then,
\begin{align*}
\textrm{dist}(p, (x_*, \varphi_{1}(x_*) ))
&\leq C (|z|^{\frac{1}{\mu}})^{\mu-1}\textrm{dist}(\widetilde{p},
(\widetilde{x}, \widetilde{\varphi_{1}}(\widetilde{x}) ) )
\leq  C (|z|^{\frac{1}{\mu}})^{\mu-1}\widetilde{x}^{\mu }\widetilde{d}\\
&\leq  C (|z|^{\frac{1}{\mu}})^{\mu-1}(|z|^{\frac{1}{\mu}})^{\mu }|z|^{\frac{1}{\mu}-1} d_1
\leq C |z|d_1,
\end{align*}
where we used \eqref{eq-distance} in estimating $\widetilde d$. Hence, with $p=(x', \varphi_1(x'))$,
\begin{equation}\label{eq-distances2}|x'-x_*|\le C|z|d_1.\end{equation}
Next, we compare $x_*$ and $x$. We write
$z=(x,y)=(|z|\cos \theta,|z|\sin\theta)$. Hence,
$$x=|z|\cos\theta,\quad\widetilde{x}=|z|^{\frac{1}{\mu}}\cos \frac{\theta}{\mu}.$$
Note $T^{-1}(\widetilde{x}, \widetilde{\varphi}_{1}(\widetilde{x}))$ is a point on $\sigma_1$ and
denote this point by $(x_*,y_*)=(r_*\cos\theta_*$, $r_*\sin\theta_*)$. The definition of $T$ implies
$$(\widetilde{x}, \widetilde{\varphi}_{1}(\widetilde{x}))
=\left(r_*^{\frac1\mu}\cos\frac{\theta_*}\mu, r_*^{\frac1\mu}\sin\frac{\theta_*}\mu\right).$$
Hence,
$$\left|\tan\frac{\theta_*}{\mu}\right|=\left|\frac{\widetilde{\varphi}_1(\widetilde x)}{\widetilde x}\right|
\le C|\widetilde x|^\mu\le C|z|,$$
and then
$$|\theta_*|\le C|z|.$$
Moreover,
$$r_*^{\frac1\mu}=(\widetilde x^2+(\widetilde\varphi_1(\widetilde x))^2)^{1/2}
=\widetilde x\left(1+\left(\frac{\widetilde{\varphi}_1(\widetilde x)}{\widetilde x}\right)^2\right)^{1/2}.$$
Then,
$$x_*=r_*\cos\theta_*=|z|\left(\cos\frac\theta\mu\right)^\mu
\left(1+\left(\frac{\widetilde{\varphi}_1(\widetilde x)}{\widetilde x}\right)^2\right)^{\mu/2}\cos\theta_*.$$
A straightforward calculation yields
$$\left|x_*-|z|\left(\cos \frac{\theta}{\mu}\right)^{\mu}\right| \le C|z|^3.$$
Note $c_0|z|>d_1> c_1|z|^2$. Then, $|\theta| \leq C\frac{d_1}{|z|}$ and hence
\begin{equation}\label{eq-distances3}|x_*-x|\le
\left||z|\left(\cos \frac{\theta}{\mu}\right)^{\mu}-|z|\cos \theta\right|+C|z|^3
\le \frac{Cd_1^2}{|z|}+C|z|d_1.\end{equation}
Therefore, \eqref{eq-distances} follows from \eqref{eq-distances1}, \eqref{eq-distances2},
and \eqref{eq-distances3}.
Denote by $l'$ and $\overline{l}$ the straight lines passing the origin and intersecting $\sigma_1$ at
$p$ and $q$, respectively. Then, the difference of their slopes can be estimated by
$$\left|\frac{\varphi_1(x')}{x'}-\frac{\varphi_1(\overline{x})}{\overline{x}}\right|\le \frac{Cd_1^2}{|z|}+C|z|d_1,$$
and a similar estimate holds for the angle between $l'$ and $\overline{l}$.
These estimates follow from a simple geometric argument and the $C^2$-bound of $\varphi_1$.
Then, with $c_1|z|^2\le d_1$,
$$|\operatorname{dist}(z,l')-\operatorname{dist}(z,\overline{l})|\le |z|\cdot \left(\frac{Cd_1^2}{|z|}+C|z|d_1\right)
\le Cd_1^2.$$
Denote by $\widehat \theta$ the angle between the line $\overline{l}$ and the tangent line of
$\sigma_1$ at $q$. Then,
$$|\operatorname{dist}(z,\overline{l})-d_1|=|d_1\cos\widehat \theta-d_1|\le Cd_1\widehat\theta^2\le Cd_1^2,$$
and hence
$$|\operatorname{dist}(z,l')-d_1|\le C d_1^2.$$
Recall that $v'$ is the solution of \eqref{eq-MainEq}-\eqref{eq-MainBoundary}
in the cone $V'$, which has the same opening angle as the tangent cone $V$.
By the explicit expressions of $v'$ in \eqref{eq-Solution-Cone}, it is straightforward to verify
\begin{equation}\label{eq-Case2.1.2}
\left|v'(z) +\log \left(\mu r \sin \frac{\arcsin \frac{d_1}{|z|}}{\mu} \right)\right|\le Cd_1.\end{equation}
By combining \eqref{eq-Case2.1.1} and \eqref{eq-Case2.1.2},
we hence have \eqref{eq-main_estimate_small} for $z\in \Omega_2\cap B_\delta$.

{\it Case 2.2.} Now we consider the general case that the map $T: z\mapsto z^{\frac1\mu}$
is not necessarily one-to-one
in $\Omega$. Take $R>0$ sufficiently small such that $T$ is
one-to-one in $D=\Omega\cap B_R$. Let $u_D$ be the solution
of \eqref{eq-MainEq}-\eqref{eq-MainBoundary} in $D$.
Then, the desired estimate for $u_D$ holds in
$\Omega_1$ and $\Omega_2$ by Case 1 and Case 2.
In the following, we denote  by $u_\Omega$ the given solution $u$ in $\Omega$.
Then, \eqref{eq-main_estimate_small} holds for $u_\Omega$ in $\Omega_1$.
We now prove \eqref{eq-main_estimate_small} holds for $u_\Omega$ in $\Omega_2$.
Since $D$ and
$\Omega$ coincide in a neighborhood of the origin, we have,
by  \eqref{sol-loc},
$$
\left|u_\Omega(z)+\log \left(\mu |z| \sin \frac{\arcsin \frac{d_1}{|z|}}{\mu} \right)\right|
\le  Cd_1+C|z|^{\frac1\mu}.$$
We need to estimate $|z|^{\frac1\mu}$.

If $\frac{1}{\mu} \geq 2$, we have $|z|^{\frac1\mu}\le |z|^2\le Cd_1$, for $z\in\Omega_2\cap B_\delta$, and
then \eqref{eq-main_estimate_small} for
$u_\Omega$ in $\Omega_2\cap B_\delta$.
For $\frac{1}{\mu} < 2$, we adopt notations in the proof of Lemma \ref{lemma-Localization}.
We take $\widetilde \mu>\mu$ sufficiently close to $\mu$ and set
$$\widetilde \theta=\theta+\frac12(\widetilde \mu-\mu)\pi.$$
By \eqref{eq-subsolution-3}, we have
\begin{equation*}
e^{2u_D}\ge\left(\frac{1}{\widetilde{\mu } |z|\sin\frac{\widetilde{\theta}}{\widetilde{\mu}}}\right)^2
\left(1+A|z|^{\frac{1}{\widetilde{\mu}}} \right)^{-2} \quad\text{in } \Omega\cap B_\delta, \end{equation*}
for $\delta$ sufficient small. Consider
$$\widehat\Omega=\Omega_2\cup\gamma_2\cup\Omega_3=\{z\in \Omega: d_1(z)< c_0|z|\}.$$
For $c_0$ small,  we have
\begin{equation*}
e^{2u_D}\ge\frac{2}{|z|^{2}}\quad\text{in }
\widehat\Omega\cap B_{\delta}, \end{equation*}
if $\delta$ is smaller.
Then, it is straightforward to verify that $u_D +\log  (1+A|z|^{2})$ is a  supersolution of
\eqref{eq-MainEq} in $\widehat\Omega\cap B_\delta.$
By Case 1, we have
\begin{equation}\label{u_2-U_1-case1}
u_\Omega \leq u_D +Cd_1\quad\text{on } \gamma_1\cap B_\delta.
\end{equation}
We set,
for two constants $a$ and $ b,$
$$\phi=a d_1 - bd_{1}^{2}.$$
Then,
$$|\Delta\phi  +a\kappa +2b |\le Cd_1, $$
where $\kappa$ is the curvature of $\sigma_1$, evaluated at the
closest point on $\sigma_1$ to $z$.
We can take positive constants $a$ and $b$ depending only on $M$ and $\mu$ such that
$$\phi>0,  \quad \Delta\phi<0
\quad\text{in }\widehat\Omega\cap B_{\delta},$$
and
$$u_\Omega \leq u_D + \phi
\quad\text{on }\gamma_1\cap B_\delta.$$
By Lemma \ref{lemma-ExteriorCone}, we have
$$|u_\Omega-u_D|\le C\quad\text{in }\Omega\cap B_\delta,$$
for $\delta$ sufficiently small. By taking $A$ large  and the maximum principle, we have
$$u_\Omega \leq u_D +\log  (1+A|z|^{2}) + \phi
\quad\text{in }\widehat\Omega\cap B_{\delta}.$$
Similarly, we obtain
$$u_D \leq u_\Omega +\log  (1+A|z|^{2}) + \phi
\quad\text{in }\widehat\Omega\cap B_{\delta},$$
and hence
$$u_\Omega = u_D +\log  (1+A|z|^{2}) + \phi
\quad\text{in }\widehat\Omega\cap B_{\delta}.$$
Note $c_1|z|^2\le d_1$ in $\Omega_2$, we get
\begin{equation}\label{u_2-U_1-case2}
|u_\Omega -u_D |\le Cd_1 \quad\text{in }\Omega_2\cap B_\delta,
\end{equation}
and hence \eqref{eq-main_estimate_small} for
$u_\Omega$ in $\Omega_2\cap B_\delta$.

We note that $\mu<1$ is not used here. We actually proved the following statement:
If
$$|u_\Omega - u_D |\le C d_1\quad\text{in }\gamma_1\cap B_\delta,$$
then
\eqref{u_2-U_1-case2}  holds in
$\Omega_2\cap B_{\delta}$.

{\it Case 3.}  We consider $z\in \Omega_3\cap B_\delta$.  We point out that we will not
need the transform $T$ in this case.

Let $q $ be the closest point on $\sigma_1$ to $z$ and set
$B_*=B_{\frac{1}{20c_1}}(q+\frac{1}{20c_1}\vec{n})$,
where $\vec{n}$ is the unit inward normal vector of $\sigma_1$ at $q$.
Note that $B_*$ is a ball tangent to $\sigma_1$ at $q$ and that
$\partial B_*$ intersects
$\gamma_2$ at two points.
Denote by $Q$ one of these intersections with the larger distance to the origin.
Then for $c_1= c_1 (M, \mu )$ large, we have $\operatorname{dist}(O, Q)<3|z|$.
With $d_1\le c_1|z|^2$, we have
\begin{align}\label{eq-expansion_cubic}
\left|\mu |z| \sin \frac{\arcsin \frac{d_{1}}{|z|}}{\mu}-d_1\right|\le C|z|
\left(\frac{d_1}{|z|}\right)^{3}\le C d_1^2
\quad\text{in }\Omega_3\cap B_\delta.
\end{align}
By what we proved in
Case 2, we have
$$|u +\log d_1|\le Cd_1\quad\text{on }\gamma_2\cap B_\delta.$$
For some positive constants $a$ and $b$, set
$$\phi=ad_1-bd_1^2.$$
Then,
$$|\Delta\phi+a\kappa+2b|\le Cd_1.$$
Let $u_*$ be the solution of \eqref{eq-MainEq}-\eqref{eq-MainBoundary} in
$B_*$.
By taking $a$ and $b$ depending only on $M$ and $\mu$, we have
$$\phi>0,  \quad \Delta\phi<0
\quad\text{in }\Omega_3\cap B_{\delta},$$
and
$$u \leq u_*+ \phi
\quad\text{on }\gamma_2\cap B_*.$$
We note that $\Omega_3\cap B_*$ consists of two parts,
$\partial B_*\cap \Omega_3$ and $\gamma_2\cap B_*$, and that
$u_*=\infty$ on $\partial B_*$.
By the maximum principle, we obtain
$$u \leq u_*+ \phi
\quad\text{in }\Omega_3\cap B_*.$$
With  $|u_{*} +\log d_1|\le Cd_1$, we have, at the fixed $z$,
$$u \leq -\log d_1 + C d_1.$$
Since we can always put a ball outside $\Omega$ and tangent to $\partial\Omega$ at
$q$ due to $\mu<1$, we get
$$u \geq -\log d_1 - C d_1.$$ Therefore, we obtain
$$|u(z)+\log d_1|\le Cd_1,$$
and hence \eqref{eq-main_estimate_small}  for $z\in \Omega_3\cap B_\delta$
by \eqref{eq-expansion_cubic}.

By combining Cases 1-3, we finish the proof of
\eqref{eq-main_estimate_small}.
\end{proof}

Now, we discuss the case when  the opening angle of the tangent cone of $\Omega$
at the origin is larger than $\pi$. We first introduce the leading term.
Let $\partial\Omega$ in a neighborhood of the origin consist of two $C^{2}$-curves
$\sigma_1$ and $\sigma_2$ intersecting at the origin at an angle $\mu\pi$, for some
constant $\mu\in (1,2)$. Define,
for any $z\in \Omega$,
\begin{equation}\label{eq-definition_f}
f_{\mu}(z)=
\begin{cases}
-\log (\mu |z|  \sin\frac{\arcsin\frac{d_{1}(z)}{|z|}}{\mu})
& \text{if } d_1(z) <d_2(z),\\
-\log(\mu |z|  \sin\frac{\theta}{\mu})
& \text{if }d_1(z) = d_2(z),\\
-\log(\mu |z| \sin\frac{\arcsin\frac{d_{2}(z)}{|z|}}{\mu})
& \text{if }d_1(z)>d_2(z),\\
\end{cases}
\end{equation}
where $d, d_1$ and $ d_2$ are the distances to $\partial\Omega, \sigma_1$ and $ \sigma_2$, respectively,
$\theta$ is the angle anticlockwise from the tangent line of $\sigma_1$ at the origin to $\overrightarrow{Oz}$.
We note that $\{z\in\Omega:\, d_1(z)=d_2(z)\}$ has a nonempty interior for $\mu\in (1,2)$ and
that $f_\mu$ is well-defined for $z$ sufficiently small.
It is straightforward to verify that $\partial\{z\in\Omega: d_1(z)< (\text{or}>) d_2(z)\}\cap \Omega$ near the origin
is a line segment perpendicular to the tangent line of $\sigma_1$ (or $\sigma_2$) at the origin. 
In fact, let $\sigma_1$ be given by a function $y=\varphi(x)\in C^{2}([0, \delta])$, for some
constant $\delta> 0$, satisfying $\varphi(0)=0$ and  \begin{equation*}
|\varphi' (x)| \leq Mx.\end{equation*} 
We now claim that $\partial\{z\in\Omega: d_1(z)< d_2(z)\}\cap \Omega$
is given by the positive vertical axis near the origin. 
To this end, we fix a point $p=(x_0,y_0)\in\Omega$ close to the origin.  If $x_0>0,$
we have 
$$\text{dist}(p, \sigma_1)\leq|p-(x_0,\varphi (x_0))|\leq  |y_0|+\frac{M}{2}x_0^2 
<\sqrt{x_0^2+y_0^2}=|p|,$$ 
if $|p|$ is small. If $x_0\leq0$, then, for any $x>0$ sufficiently small, 
$$|p|=\sqrt{x_0^2+y_0^2}<\sqrt{(x_0-x)^2+(y_0-\varphi(x))^2}=|p-(x,\varphi(x))|.$$
The finishes the proof of the claim. 

\begin{theorem}\label{thrm-LargeAngles}
Let $\Omega$ be a bounded domain in $\mathbb R^2$ and
$\partial\Omega\cap B_{r_0}$  consist of two $C^{2}$-curves
$\sigma_1$ and $\sigma_2$ intersecting at the origin at an angle $\mu\pi$, for some
constants $\mu\in (1,2)$ and $r_0>0$. Suppose $ u \in C^{2}(\Omega)$ is  a solution of
\eqref{eq-MainEq}-\eqref{eq-MainBoundary}.
Then, for any $z\in\Omega\cap B_{\delta}$,
$$|u(z)-f_\mu(z)|\le Cd(z),$$
where
$f_\mu$ is the function defined by \eqref{eq-definition_f},
and $ \delta$ and $C$ are positive constants depending only on
$\mu$, $r_0$ and the $C^2$-norms of $\sigma_1$ and $\sigma_2$.
\end{theorem}

\begin{proof}
We proceed similarly as in the proof of Theorem \ref{thrm-SmallAngles} and adopt the same notations.
We denote by $M$ the $C^{2}$-norms of $\sigma_1$ and $\sigma_2$, and define
$\Omega_1$, $\Omega_2$, $\Omega_3$ and $\gamma_1, \gamma_2$
by \eqref{eq-Definition_Omega} and \eqref{eq-Definition_gamma}, respectively,
where $c_0$ and $c_1$ are appropriately chosen constants
with $c_0 <\frac{1}{2} \arctan \frac{1}{4}$.
Consider \textit{T}: $z\mapsto z^{\frac{1}{\mu}}$.

We fix a point $z\in \Omega\cap B_\delta$, for some $\delta$ sufficiently small.
Without loss of generality, we assume
$d_{1}=d_{1}(z)=d(z)\leq d_2=d_{2}(z)$.

{\it Case 1.} We consider $z\in \Omega_1\cap B_\delta$.

Set
$\Omega_{+}= \Omega\cap B_{\delta}$
and let $u_+$
be the solution of \eqref{eq-MainEq}-\eqref{eq-MainBoundary}
in $\Omega_{+}$.
We take $\delta$ small so that $T$ is one-to-one on $\Omega_+$.
Set
$\widetilde{\Omega}_{+}=T(\Omega_+)$ and let $\widetilde u_+$ be the solution of
\eqref{eq-MainEq}-\eqref{eq-MainBoundary} in $\widetilde\Omega_+$.
By \eqref{eq-estimates_on_new_curve}, the curve $\widetilde\sigma$ given by
$\widetilde y=\widetilde \varphi(\widetilde x)$
satisfies
$$-\widetilde M|\widetilde x|^{1+\mu}\le |\widetilde\varphi(\widetilde x)|\le \widetilde M|\widetilde x|^{1+\mu}.$$
We note here $1+\mu>2$.
Theorem \ref{thrm-C-2,alpha-expansion} implies, for $\widetilde z$ close to the origin,
\begin{equation}\label{case1-lower1}
\widetilde{u}_+(\widetilde{z})\leq -\log\widetilde{ d}_{1}+\frac12\kappa_1\widetilde d_1 +C\widetilde{d}_{1}^{\mu},
\end{equation}
and
\begin{equation}\label{case1-upper1}
\widetilde{u}_+(\widetilde{z})\geq -\log \widetilde{ d}_{2}+\frac12\kappa_2\widetilde d_2 -C\widetilde{d}_{2}^{\mu},
\end{equation}
where $\widetilde{d}_1$ and $\widetilde{d}_2$ are the distances from $\widetilde{z}$ to the curves
$\widetilde{y}=\widetilde{M}|\widetilde{x}|^{1+\mu}$ and
$\widetilde{y}=-\widetilde{M}|\widetilde{x}|^{1+\mu}$, respectively, and
$\kappa_1$ and $\kappa_2$ are the curvatures of the curves
$\widetilde{y}=\widetilde{M}|\widetilde{x}|^{1+\mu}$ and
$\widetilde{y}=-\widetilde{M}|\widetilde{x}|^{1+\mu}$, respectively.
Recall from the proof of Theorem \ref{thrm-SmallAngles} that, for $c_0|z|<d_1$,
$$|\log \widetilde d_i-\log \widetilde y|\le Cd_1,$$
and
$$\widetilde d_i^\mu\le  Cd_1.$$
Moreover,
$$|\kappa_i|\le C|\widetilde{z}|^{\mu-1}=C|z|^{\frac{\mu-1}{\mu}}\le Cd_1^{\frac{\mu-1}{\mu}}.$$
Therefore, \eqref{case1-lower1} and \eqref{case1-upper1} imply
$$|\widetilde u_+(\widetilde z)+ \log  \widetilde y|\le Cd_1.$$
This is the same as \eqref{eq-estimate_u_+}. The rest of the proof for Case 1 is identical to that in the proof of
Theorem \ref{thrm-SmallAngles}.

{\it Case 2.} We consider $z\in \Omega_2\cap B_\delta$.

Arguing similarly as in the proof of
Theorem \ref{thrm-SmallAngles}, we have
\begin{equation}\label{taylor-expansion_of_curve2}
|\widetilde{\varphi}(\widehat{x})-\widetilde{\varphi}(\widetilde{x}')
-\widetilde\varphi'(\widetilde x')(\widehat{x}-\widetilde{x}')|
\leq K (|z|^{1-\frac{1}{\mu}}+ |\widehat{x} -\widetilde{x}'|^{\mu -1} )
(\widehat{x} -\widetilde{x}' )^2.
\end{equation}
This plays a similar role as \eqref{taylor-expansion_of_curve}.
Then,  we have
\begin{equation*}
\widetilde{u}(\widetilde{z})\leq -\log \widehat{ d}_{1}
+\frac12\kappa_{1}\widehat{d}_{1}+C\widehat{d}_{1}^{\mu},
\end{equation*}
and
\begin{equation*}
\widetilde{u}(\widetilde{z})\geq -\log \widehat{ d}_{2}
+\frac12\kappa_{2}\widehat{d}_{2} -C\widehat{d}_{2}^{\mu},
\end{equation*}
where $\widehat{d}_1$ is the distance from $\widetilde{z}$ to the curve
$$\widehat{y}=\widetilde{\varphi}(\widetilde{x}')
+\widetilde\varphi'(\widetilde x')(\widehat{x}-\widetilde{x}')
+ K (|z|^{1-\frac{1}{\mu}}+ |\widehat{x} -\widetilde{x}'|^{\mu -1} )
(\widehat{x} -\widetilde{x}' )^2,$$
and $\widehat{d}_2$ is the distance from $\widetilde{z}$ to the curve
$$\widehat{y}=\widetilde{\varphi}(\widetilde{x}')
+\widetilde\varphi'(\widetilde x')(\widehat{x}-\widetilde{x}')
- K (|z|^{1-\frac{1}{\mu}}+ |\widehat{x} -\widetilde{x}'|^{\mu -1} )
(\widehat{x} -\widetilde{x}' )^2.$$
Then, we proceed similarly as in Case 2 in the proof of
Theorem \ref{thrm-SmallAngles}.

 {\it Case 3.} We consider $z\in \Omega_3\cap B_\delta$.

We take $q\in \sigma_1$ with the least distance to $z$,
and denote by $l$ the tangent line of $\sigma_1$ at $q$.
We put $q$ at the origin of the line $l$.
A portion of $\sigma_1$ near $q$, including the part from the origin to $q$, can be
expressed as a $C^{2}$-function $\varphi$ in $(-s_0, s_0)$, with $\varphi(-s_0)$
corresponding to the origin in $\mathbb R^2$ and $\varphi(0)$ corresponding to $q$, i.e., $\varphi(0)=0$.
Then,
\begin{equation}\label{eq-boundaryC1alpha-Version2}
|\varphi(s)|\le \frac12M|s|^{2}\quad\text{for any }s\in (-s_0,s_0).\end{equation}
In the present case, $M$ is uniform, independent of $z$; however, $s_0$ depends on $z$.
We should first estimate $s_0$ in terms of $d_2$.
We note, for $d_2$ sufficiently small,
\begin{equation}\label{eq-Estimate-d_2}\frac12|z|\sin\frac{(2-\mu)\pi}{2}\le d_2\le |z|.\end{equation}
By the triangle inequality and
\eqref{eq-boundaryC1alpha-Version2}, we have
$$s_0\le \frac12Ms_0^{2}+|z|+d_1,$$
and
$$s_0\ge -\frac12Ms_0^{2}+|z|-d_1.$$
Then, $s_0/|z|\to 1$ as $|z|\to 0$.
We take $|z|$ sufficiently small such that $s_0\ge 2|z|/3$.

By taking $|z|$ sufficiently small, \eqref{eq-Estimate-d_2} implies
$$B_{r_1}(q-r_1\vec{n})\cap \Omega=\emptyset,$$
where $\vec{n}$ is the  unit inward normal vector of $\sigma_1$ at $q$ and
$$r_1=\frac12|z|\sin\frac{(2-\mu)\pi}{8}.$$
Let  $v_{r_1, q-r_1\vec{n}}$ be the solution of
\eqref{eq-MainEq}-\eqref{eq-MainBoundary} in
$\mathbb R^2\setminus B_{r_1}(q-r_1\vec{n})$,
given by  \eqref{eq-solution-outside}.
By the maximum principle, we have
$$u \geq v_{r_1,q-r_1\vec{n}}\quad\text{in }\Omega.$$
Hence,
\begin{equation}\label{expansion_d_1/2}
u\ge -\log d_1- C|z| \quad\text{in }\Omega_3\cap B_{\delta}.
\end{equation}

By taking $R = R(M,\mu)$ small, we have
$$\textrm{dist}(z', \sigma_1) \leq \frac{1}{2}\textrm{dist}(z',\partial B_R(q-R\vec{n})).$$
By what we proved in Case 2, we get
$$ |u +\log d_1 |\le Cd_1\quad\text{on }\gamma_2\cap B_\delta.$$
Combining with \eqref{expansion_d_1/2}, we have, for $|z|$ sufficient small,
$$u \geq v_{R,q-R\vec{n}}\quad\text{in }\Omega_3\cap
\partial B_{3|z|} (z).$$
Set
$$\phi=a d_1 - bd_{1}^{2}.$$
We can take two positive constants $a$ and $b$ depending only the geometry of $\Omega$ such that
$$\phi>0,  \quad \Delta\phi<0
\quad\text{in }\Omega_3\cap \partial B_{\delta},$$ and
$$ v_{R,q-R\vec{n}}\leq u + \phi\quad\text{on }\gamma_2\cap B_{\delta}.$$
By the maximum principle, we obtain
$$ v_{R,q-R\vec{n}}\leq u + \phi\quad\text{in }\Omega_3\cap B_{3|z|} (z).$$
By
$$|v_{R,q-R\vec{n}} +\log d_1 |\le Cd_1,$$
we have
$$u(z) \geq -\log d_1 - C d_1.$$
Since we can always put a ball inside $\Omega$ and tangent to $\partial\Omega$ at
$q$ due to $\mu>1$, we get
$$u(z) \leq -\log d_1 + C d_1.$$ Therefore,
$$|u(z) +\log d_1 |\le Cd_1,$$
and hence
$$\left|u(z) +\log \left(\mu r \sin \frac{\arcsin \frac{d_{1}}{r}}{\mu} \right) \right|\le Cd_1.$$
This is the desired estimate for $z\in \Omega_3\cap B_\delta$.
\end{proof}

\begin{remark}
We point out that the estimates in Theorem \ref{thrm-SmallAngles} and
Theorem \ref{thrm-LargeAngles} are local;
namely, they hold in $\Omega$ near the origin, independent of $\Omega$ away from the origin.
\end{remark}

\begin{remark}
The function $f_\mu$ in Theorem \ref{thrm-SmallAngles} and
Theorem \ref{thrm-LargeAngles} is locally Lipschitz since it involves the distance function, which
is Lipschitz, and  is piecewise $C^2$. In fact, $f_\mu$ is $C^2$ except along a curve
given by $d_1=d_2$ for $\mu\in (0,1)$ and except along two curves for $\mu\in (1,2)$.
On the other hand, we can replace $f_\mu$ by a function which is $C^2$ in $\Omega\cap B_{\delta}$
and maintain the same estimates as in Theorem \ref{thrm-SmallAngles} and
Theorem \ref{thrm-LargeAngles}.
\end{remark}

\begin{remark}
With a slightly more complicated argument, we can prove the following estimate:
if $\sigma_1$ and $\sigma_2$ are  $C^{1,\alpha}$-curves, for some $\alpha\in (0,1)$, then
for any $z\in \Omega\cap B_\delta$,
$$
\left|u(z)-f_\mu(z)\right| \leq Cd^\alpha(z),$$
where $f_\mu$ is given by \eqref{eq-definition_f0} for $\mu\in (0,1]$ and
by \eqref{eq-definition_f} for $\mu\in (1,2)$,
and $\delta$ and $C$ are positive constants depending only on the geometry of $\partial\Omega$.
This estimate can be viewed as a generalization of Theorem \ref{thrm-C-1,alpha-expansion}.
\end{remark}

\end{document}